\documentclass[11pt]{article}

\usepackage[utf8]{inputenc}
\usepackage[T1]{fontenc}
\usepackage{amsmath,amssymb,amsthm}
\usepackage{longtable}
\usepackage{booktabs}
\usepackage{enumitem}
\usepackage{graphicx}
\usepackage{subcaption}
\usepackage{pdflscape}
\usepackage[margin=1in]{geometry}
\usepackage{aliascnt}

\usepackage{xcolor}
\usepackage[colorlinks,citecolor=blue!70!black,linkcolor=red!60!black]{hyperref}
\usepackage[nameinlink]{cleveref}

\theoremstyle{definition}
\newtheorem{theorem}{Theorem}[section]
\newaliascnt{definition}{theorem}
\newtheorem{definition}[definition]{Definition}
\aliascntresetthe{definition}
\newaliascnt{lemma}{theorem}
\newtheorem{lemma}[lemma]{Lemma}
\aliascntresetthe{lemma}
\newaliascnt{corollary}{theorem}
\newtheorem{corollary}[corollary]{Corollary}
\aliascntresetthe{corollary}
\newaliascnt{proposition}{theorem}

\aliascntresetthe{proposition}
\newaliascnt{remark}{theorem}
\newtheorem{remark}[remark]{Remark}
\aliascntresetthe{remark}
\newaliascnt{example}{theorem}
\newtheorem{example}[example]{Example}
\aliascntresetthe{example}
\newtheorem{maintheorem}{Theorem}

\newtheorem*{namedquestion}{Question}
\newtheorem{question}{Question}

\crefname{maintheorem}{theorem}{theorems}
\Crefname{maintheorem}{Theorem}{Theorems}
\crefname{theorem}{theorem}{theorems}
\Crefname{theorem}{Theorem}{Theorems}
\crefname{definition}{definition}{definitions}
\Crefname{definition}{Definition}{Definitions}
\crefname{lemma}{lemma}{lemmas}
\Crefname{lemma}{Lemma}{Lemmas}
\crefname{corollary}{corollary}{corollaries}
\Crefname{corollary}{Corollary}{Corollaries}
\crefname{proposition}{proposition}{propositions}
\Crefname{proposition}{Proposition}{Propositions}
\crefname{remark}{remark}{remarks}
\Crefname{remark}{Remark}{Remarks}
\crefname{question}{question}{questions}
\Crefname{question}{Question}{Questions}
\crefname{appendix}{appendix}{appendices}
\Crefname{appendix}{Appendix}{Appendices}
\crefname{example}{example}{examples}
\Crefname{example}{Example}{Examples}
\crefname{subfigure}{figure}{figures}
\Crefname{subfigure}{Figure}{Figures}

\title{Amphichiral Knots: Odd Braid Index and Symmetry Classification in the Three-Braid Case}
\author{\textsc{Hyungseok Jung}}
\date{}

\begin{document}
\maketitle

\begin{abstract}
We study how amphichirality of a knot constrains its braid index and, in the
smallest nontrivial case, the combinatorics of its braid words. Using the
Dynnikov--Prasolov resolution of the Jones conjecture, we show
the braid index of an amphichiral knot is odd. This answers,
in the negative, a question of Stoimenow on amphichiral knots of even braid
index. We then classify prime amphichiral knots of braid index~$3$. Every
amphichiral knot of braid index~$3$ is alternating and admits a minimal
$3$-braid representative in a standard form encoded by a word~$c$. Building on
the Birman--Menasco classification of closed
$3$-braids and the Murasugi normal form, we describe a dihedral action on~$c$
under which the mirror and mirror-reverse operations are realized by a rotation
and a reflection. For a standard form whose closure is a prime knot
of braid index~$3$, this yields a complete criterion:
$\widehat{\beta_c}$ is amphichiral if and only if $c$ is a palindrome
or has odd period, together with a determination of the precise symmetry
type in terms of these properties and the presence of a non-degenerate flype.
\end{abstract}

\section{Introduction}

A knot $K$ is said to be \emph{amphichiral} if, as an unoriented knot, it is
isotopic to its mirror image $mK$. Amphichirality is among the most basic
symmetry properties of a knot, and it is subtly intertwined with orientation.
Once an orientation is fixed, amphichirality refines into three types, namely
\emph{positive}, \emph{negative}, and \emph{fully} amphichiral knots. These three
types are distinguished according to whether $K$ is equivalent to its mirror
image and whether that equivalence can be taken compatibly with reversing the
orientation. Our aim in this paper is to understand how this symmetry interacts
with a purely algebraic invariant, the \emph{braid index}, and, in the low-index
case, with the combinatorics of the braid words that represent a knot.

By Alexander's theorem~\cite[Theorem~2.3]{KasselTuraev}, every oriented link in
$\mathbb{R}^3$ is isotopic to a closed braid, so to each knot there corresponds a
braid word. This makes it convenient to study a knot algebraically, through its
braid word rather than the knot diagram directly, which is especially effective
for families such as torus knots. For links of braid index~$3$ this algebraic viewpoint is now
essentially complete. Birman and Menasco~\cite[The Classification Theorem
(Version~1)]{BM1993} classified the conjugacy classes of $3$-braid
representatives, and combined with the Murasugi normal form for closed
$3$-braids~\cite[Theorem~6.1]{DasbachLowrance2011} and the resolution of the Jones
conjecture by Dynnikov and Prasolov~\cite[Theorem~9]{DP2013}, the symmetry
analysis of closed $3$-braids becomes entirely explicit.

The paper is built around three main results. The first, proved in
\Cref{sec:partI}, constrains the braid index of any amphichiral
knot.

\begin{maintheorem}\label{thm:A}
The braid index of an amphichiral knot is odd.
\end{maintheorem}

\noindent
The proof rests on a consequence of the Jones conjecture: the exponent sum of a
minimal-strand braid representative is a link invariant. Amphichirality forces
this invariant to vanish, and a parity argument on the underlying permutation
then yields the conclusion. As a corollary, the $(2,k)$-torus knots are chiral,
that is, non-amphichiral (\Cref{cor:torus-chiral}).

\Cref{thm:A} answers a question of Stoimenow. In the course of analysing
braid-index criteria, Stoimenow \cite[Question~4]{StoimenowMFW} asked the following.
\begin{namedquestion}
Is there an amphichiral knot $K$ of even braid index?
\end{namedquestion}
\noindent(The original question used the term `achiral', but since this has the
same meaning as `amphichiral', it has been changed to the latter expression here.)
He observed that such a knot would furnish a counterexample to the then-open Jones conjecture
on the writhe of minimal-strand braid representatives. Since the Jones conjecture
has since been established by Dynnikov and Prasolov~\cite[Theorem~9]{DP2013}, combining their
theorem with the parity argument above answers this question in the negative and
definitively. That is, no amphichiral knot has even braid index, and so \Cref{thm:A}
resolves Stoimenow's question.

In Part~II (\Cref{sec:partII}) we classify prime amphichiral knots of braid
index~$3$. For a word of positive integers
$c=(a_1,b_1,\ldots,a_n,b_n)$, consider the $3$-braid in
\emph{alternating standard form}
\[
  \beta_c
  = \sigma_1^{a_1}\sigma_2^{-b_1}\cdots\sigma_1^{a_n}\sigma_2^{-b_n},
  \qquad a_i, b_i \ge 1,
\]
and let $D_c$ denote its standard closed-braid diagram. The diagram $D_c$ is
alternating. More importantly, the alternating standard form loses no
generality for the knots under consideration:
\Cref{thm:amphi-implies-alt} shows that every amphichiral knot of braid
index~$3$ admits a minimal $3$-braid representative in alternating standard
form and is therefore alternating. Thus the classification below applies to
every prime amphichiral knot of braid index~$3$ after choosing such a
representative; no separate alternating or reduced-diagram hypothesis is
required. Composite knots are not included in \Cref{thm:B,thm:C} or in the
appendices.

The classification is stated in two theorems: \Cref{thm:B} gives the general
criterion for amphichirality, and \Cref{thm:C} refines it into the precise symmetry
type.

\begin{maintheorem}\label{thm:B}
Let \[\beta_c=\sigma_{1}^{a_1}\sigma_{2}^{-b_1}\cdots
\sigma_{1}^{a_n}\sigma_{2}^{-b_n},
\qquad a_i,b_i\ge1,
\]
be a braid in alternating standard form, and set $K=\widehat{\beta_c}$. Suppose that $K$ is a prime knot of braid index~$3$. Then 
\[
K\ \text{is amphichiral}
\quad\Longleftrightarrow\quad
c\ \text{is a palindrome or has odd period.}
\]
\end{maintheorem}

\begin{maintheorem}\label{thm:C}
Retain the hypotheses of \Cref{thm:B}. Then exactly one of the following two cases
occurs, and in each case the symmetry type of $K$ is determined as shown.
\begin{enumerate}[label=\textup{(\Alph*)}]
\item \emph{$\beta_c$ represents $K$ by a unique conjugacy class of $3$-braids.}
In this case,
  \begin{enumerate}[label=\textup{(\roman*)}]
  \item $K$ is fully amphichiral $\iff$ $c$ is a palindrome and has odd period;
  \item $K$ is positive amphichiral but not negative amphichiral $\iff$ $c$ has odd period but is not a palindrome;
  \item $K$ is negative amphichiral but not positive amphichiral $\iff$ $c$ is a palindrome but does not have odd period.
  \end{enumerate}
\item \emph{The conjugacy class of $\beta_c$ admits a non-degenerate flype.} In this
case, $K$ is amphichiral if and only if $c$ is of the form $(k,1,1,k)$ up to cyclic
permutation with $k\ge 2$; in particular, $c$ is a palindrome, and $K$ is fully
amphichiral.
\end{enumerate}
\end{maintheorem}

\noindent The mechanism underlying both theorems is a dihedral action on the word $c$ (\Cref{lem:dihedral}). Let $\rho$ denote the rotation that shifts every entry of $c$ one position to the left, moving the first entry to the end. Under this action, $\rho$ realizes the mirror operation on $\widehat{\beta_c}$, while the corresponding reflection realizes the mirror-reverse operation. Thus the symmetry properties of $\widehat{\beta_c}$ are encoded by the rotational and reflectional symmetries of $c$. The group-theoretic formulation uses a homomorphism from the dihedral group to the four-element operation group $G$, and identifies the actual symmetries of $K$ with the stabilizer of the orbit of $c$ under the induced action. The precise quotient and stabilizer descriptions are developed in \Cref{sec:partII}.

The paper is organized as follows. \Cref{sec:def} fixes the notation and
terminology used throughout: the braid group $B_n$ with its permutation
homomorphism $\pi\colon B_n\to S_n$, a criterion for when the closure of a
$3$-braid is a knot, an algebraic description of amphichirality in its positive,
negative, and fully amphichiral types, and the mirror and reverse operations on
braid words. \Cref{sec:partI} proves \Cref{thm:A} and its corollary. \Cref{sec:partII} develops the dihedral action (\Cref{lem:dihedral}) and assembles the classification: the case of a unique conjugacy class is treated in \Cref{thm:unique-amphi}, and the flype case in \Cref{lem:flype-family} and \Cref{thm:flype-amphi}; \Cref{thm:B,thm:C} are read off from these.

The appendices make the classification explicit
using the data in~\cite{knotinfo}. \Cref{tab:bi3-cn12} (\Cref{app:bi3-cn12}) lists
every prime amphichiral knot of braid index~$3$ with crossing number at most~$12$, and
\Cref{tab:bi3-cn14} (\Cref{app:bi3-cn14}) those with crossing number~$14$, recording
for each the word $c$, a braid representative, and the resulting symmetry
type. Together they list every prime amphichiral knot of braid index~$3$ up to crossing
number~$14$, each of which is alternating by
\Cref{thm:amphi-implies-alt}. \Cref{tab:bi5-cn10} (\Cref{app:bi5-cn10})
lists the prime amphichiral knots of braid index~$5$ with crossing number at most~$10$;
every entry has odd braid index and no knot of even braid index appears, in
agreement with \Cref{thm:A}.

Since amphichiral knots always have odd braid index, the natural next case is
braid index~$5$. This suggests the following further questions.
\begin{question}
What characteristics do the words of amphichiral knots of braid
index~$5$ possess?
\end{question}

\noindent
When the braid index is $3$, the Birman--Menasco classification made this kind of
symmetry analysis tractable, since it reduces the classification of closed
$3$-braids to a small, explicit list of conjugacy-class exceptions. For braid
index~$4$, however, no comparably explicit classification of closed $4$-braids
is currently known, so even the basic combinatorial framework used in
\Cref{sec:partII} is unavailable. Braid index~$5$ presents an even greater
obstacle: the corresponding braid group has more generators and a
substantially richer conjugacy structure, so the knots themselves grow rapidly
in complexity, and a full classification along the lines of \Cref{thm:B,thm:C}
appears, at present, to be out of reach.

\begin{question}\label{q:invariants-from-word}
If amphichirality can be determined directly from a braid word, can other knot
invariants, such as the determinant $\det(K)$, the signature $\sigma(K)$, the
Seifert genus $g(K)$, or the unknotting number $u(K)$, likewise be detected or
computed from the braid word, and what algebraic criteria are required to do so?
\end{question}

\noindent
An affirmative answer to this question would be of considerable practical
value: it would translate a knot's complicated topological structure into
purely algebraic data carried by its braid word, making it possible to
compute knot invariants algebraically, directly from the word, rather than
through geometric or diagrammatic arguments.

\section*{Acknowledgements}
The author is deeply grateful to his advisor, Sanghoon Kwak, for invaluable
guidance and encouragement throughout this work. The author also thanks his
colleagues Jihun Kwak, Jiwoo Park, and Jinseon Lee for many helpful discussions.
The author was supported by the Young Scientist Grant (RS-2026-25481514) of the
National Research Foundation of Korea (NRF), funded by the Korea government
(MSIT).

\section*{AI Acknowledgements}
In preparing this manuscript, the author used Claude (Opus~4.8), developed by Anthropic, as an assistive tool for translating the author's Korean drafts into English and for assisting with the preparation of the \LaTeX{} source. The author also used the tool in the computational analysis of braid representatives. More specifically, braid-word data obtained from KnotInfo~\cite{knotinfo} were provided to the tool, which assisted in finding and checking sequences of braid relations and Markov moves that transform the given representatives into forms in which their symmetry properties are more readily visible. These computations were used in identifying braid representatives associated with positive, negative, or full amphichirality. The tool was further used to organize the resulting braid words and symmetry classifications into the tables recorded in the appendices; systematically arranging the data in this tabular form helped the author to notice the recurring patterns in the words~$c$, such as the palindrome and odd-period phenomena underlying \Cref{thm:B,thm:C}, that motivated and shaped the classification developed in this paper. The author independently reviewed and verified the resulting braid-word transformations and tabulated data, edited all AI-generated text, and takes full responsibility for the mathematical arguments and the contents of this paper.

\section{Definitions}\label{sec:def}
\noindent In this section we set up the notation and terminology for the rest of
the paper. We recall the braid group $B_n$ and its closure, and use the
permutation homomorphism $\pi\colon B_n\to S_n$ to characterize when the closure
of a $3$-braid is a knot, which occurs exactly when $\pi(\beta)$ is a $3$-cycle. We then
describe amphichirality algebraically and record the mirror and reverse
operations on braid words, which realize the geometric mirror and mirror-reverse
on the closure.

\begin{definition}[Braid group {\cite[definition 1.1]{KasselTuraev}}]\label{def:braid}
\leavevmode
\begin{enumerate}[label=\textup{(\alph*)}]
\item The \emph{Artin braid group} $B_{n}$ is the group generated by the $n-1$
generators $\sigma_{1},\dots,\sigma_{n-1}$ subject to the braid relations. It admits the
presentation
\[
B_{n}=\Big\langle \sigma_{1},\dots,\sigma_{n-1}\ \Big|\
\begin{aligned}
&\sigma_{i}\sigma_{j}=\sigma_{j}\sigma_{i}, && |i-j|\ge 2,\\
&\sigma_{i}\sigma_{i+1}\sigma_{i}=\sigma_{i+1}\sigma_{i}\sigma_{i+1}, && 1\le i\le n-2
\end{aligned}
\Big\rangle .
\]
Let $\pi\colon B_{n}\to S_{n}$ be defined by $\pi(\sigma_{i})=s_i$, the transposition
$(i\,i+1)$. Then $\pi$ is a surjective homomorphism.
\item The \emph{closure} of a braid $\beta$, denoted $\widehat{\beta}$, is formed by
connecting the top endpoints of the strands to the corresponding bottom
endpoints, thereby producing a closed link.
\end{enumerate}
\end{definition}

\begin{figure}[ht!]
    \centering
    \includegraphics[width=0.4\linewidth]{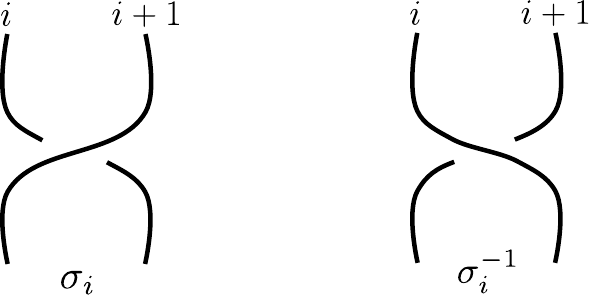}
\caption{The Artin generators acting on strands $i$ and $i+1$. The
left-hand diagram represents $\sigma_i$, and the right-hand diagram represents
$\sigma_i^{-1}$. With all strands oriented from top to bottom, $\sigma_i$ is the
positive crossing (sign $+1$), whereas $\sigma_i^{-1}$ is the negative crossing
(sign $-1$).}
\label{fig:generators}
\end{figure}

\noindent The number of components of the closure $\widehat{\beta}$ of
$\beta\in B_{n}$ equals the number of disjoint cycles of the permutation
$\pi(\beta)\in S_{n}$: each cycle of $\pi(\beta)$ traces out one closed loop when
the endpoints are joined. Hence $\widehat{\beta}$ is a knot (a one-component
link) if and only if $\pi(\beta)$ is a single $n$-cycle. In particular, a
$3$-braid $\beta\in B_{3}$ can be realized as a knot only when $\pi(\beta)$ is a
$3$-cycle.

\begin{definition}[Amphichiral knot]\label{def:amphi}
For a knot $K$ there are two operations.
\begin{enumerate}[label=\textup{(\alph*)}]
\item \emph{Mirror} ($m$): reversal of the orientation of $S^{3}$; in a
diagram this flips every crossing ($\sigma_{i}\mapsto\sigma_{i}^{-1}$). The orientation of
the knot itself is preserved.
\item \emph{Reverse} ($r$): the same knot with the opposite orientation.
\end{enumerate}

\begin{figure}[ht!]
\centering
\includegraphics[width=0.7\linewidth]{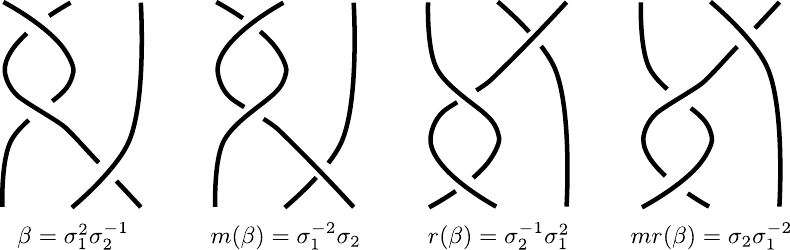}
\caption{The mirror and reverse operations on the braid
$\beta=\sigma_1^{2}\sigma_2^{-1}$, drawn with all strands oriented from top to
bottom. From left to right, the four diagrams represent $\beta$, $m(\beta)$,
$r(\beta)$, and $mr(\beta)$, respectively. The mirror $m$ flips every crossing
($\sigma_i\mapsto\sigma_i^{-1}$) while preserving the order and strand
positions; the reverse $r$ reads the word from bottom to top, reversing only the
order of the crossings while keeping each sign and strand position; and $mr$
combines the two operations.}
\label{fig:mirror-reverse}
\end{figure}

The mirror and reverse operations generate the four-element
\emph{operation group}
\[
G:=\{1,m,r,mr\}
\cong
\mathbb{Z}/2\mathbb{Z}\times\mathbb{Z}/2\mathbb{Z}.
\]

\smallskip
\noindent
The corresponding symmetry conditions are summarized in the following table.
\begin{center}
\footnotesize
\renewcommand{\arraystretch}{1.25}
\setlength{\tabcolsep}{4pt}
\begin{tabular}{c||c|c|c|c}
\hline
Operation & Orient. of $K$ & Orient. of $S^3$ & Relation & Property\\
\hline\hline
$r$ & reversed & preserved & $K=rK$ & invertible\\
\hline
$m$ & preserved & reversed & $K=mK$ & positive amphichiral\\
\hline
$mr$ & reversed & reversed & $K=mrK$ & negative amphichiral\\
\hline
$m,r,mr$ & \multicolumn{3}{c|}{all of the above} & fully amphichiral\\
\hline
\end{tabular}
\end{center}

\smallskip
\noindent
The \emph{symmetry subgroup} of $K$ is
\[
\mathrm{Sym}(K)
:=
\{\,g\in G:g(K)=K\,\}
\le G,
\]
where $g(K)=K$ is understood as equality of oriented knot types in
$S^3$. This is a subgroup because the operations fixing $K$ are closed
under composition and inverses.

We say that $K$ is \emph{positive amphichiral} if
\[
m\in\mathrm{Sym}(K),
\]
and \emph{negative amphichiral} if
\[
mr\in\mathrm{Sym}(K).
\]
The knot is \emph{invertible} if $r\in\mathrm{Sym}(K)$, and it is
\emph{fully amphichiral} if
\[
\mathrm{Sym}(K)=G.
\]
Thus a fully amphichiral knot is simultaneously positive amphichiral,
negative amphichiral, and invertible.

We call $K$ \emph{amphichiral} if
\[
m\in\mathrm{Sym}(K)
\qquad\text{or}\qquad
mr\in\mathrm{Sym}(K),
\]
that is, if $K$ agrees with its mirror after possibly reversing the
orientation of the knot. Throughout the classification statements,
the phrases ``positive but not negative'' and ``negative but not
positive'' denote the two non-fully-amphichiral cases. In the examples
and appendix tables, the shorter labels ``positive'' and ``negative''
are used for these two exclusive symmetry types, while ``fully'' is
listed separately.
\end{definition}

\noindent Next, for the braid word
$\beta=\sigma_{i_1}^{\varepsilon_1}\sigma_{i_2}^{\varepsilon_2}\cdots
\sigma_{i_k}^{\varepsilon_k}$ of a knot, the two operations act as follows.
\begin{enumerate}[label=\textup{(\alph*)}]
\item \emph{Mirror:} $m(\beta)=\sigma_{i_1}^{-\varepsilon_1}
\sigma_{i_2}^{-\varepsilon_2}\cdots\sigma_{i_k}^{-\varepsilon_k}$. Flipping the sign of
every exponent yields the braid whose closure is the mirror image; thus
$\widehat{m(\beta)}=m(\widehat{\beta})$.
\item \emph{Reverse:} $r(\beta)=\sigma_{i_k}^{\varepsilon_k}\cdots
\sigma_{i_2}^{\varepsilon_2}\sigma_{i_1}^{\varepsilon_1}$, i.e.\ the word read backwards.
Reversing the orientation and reading the braid from bottom to top preserves the
sign and the strand position of each crossing while reversing only their order.
Hence $\widehat{r(\beta)}=r(\widehat{\beta})$.
\end{enumerate}

For a detailed illustration of the mirror and reverse operations on a braid word and their effects on its closure, see \Cref{fig:mirror-reverse}.

\section{Part I: Amphichiral Knots and the Braid Index}\label{sec:partI}

\noindent In this part we prove \Cref{thm:A}. We first record the relevant numerical
invariants of a braid and the Jones conjecture, deduce that minimal-strand
representatives have a well-defined exponent sum, and then run the parity argument.

\begin{definition}[Algebraic length, writhe, minimal-strand representative]\label{def:length}
\leavevmode
\begin{enumerate}[label=\textup{(\alph*)}]
\item Let $\beta=\sigma_{i_1}^{\varepsilon_1}\sigma_{i_2}^{\varepsilon_2}\cdots
\sigma_{i_k}^{\varepsilon_k}\in B_{n}$ be a braid word, where $i_j\in\{1,\dots,n-1\}$
and $\varepsilon_j=\pm1$. We denote by $e$ the homomorphism from $B_{n}$ to $\mathbb{Z}$ defined by
$e(\sigma_{i})=1$. In other words,
\[
e(\beta)=\sum_{j=1}^{k}\varepsilon_j
\]
is the \emph{algebraic number} (algebraic length) of crossings of a braid
diagram representing $\beta$.
\item Let $D$ be an oriented link diagram. The \emph{writhe} of $D$ is
\[
w(D)=\sum_{x\in\operatorname{Cross}(D)}\operatorname{sign}(x).
\]
\item The \emph{braid index} of a link $L$ is the least integer $m$ for which $L$
is isotopic to the closure of some braid $\beta'\in B_{m}$. A braid attaining this
minimum is called a \emph{minimal-strand representative} of $L$.

\end{enumerate}
In a braid all strands run from top to bottom; hence $\sigma_{i}$ has sign $+1$ and
$\sigma_{i}^{-1}$ has sign $-1$. Consequently, if $D_\beta$ denotes the standard
closed-braid diagram of $\widehat{\beta}$, then
\[
w(D_\beta)=e(\beta).
\]
\end{definition}

\begin{theorem}[Jones conjecture (Dynnikov-Prasolov Theorem) {\cite[Theorem~9]{DP2013}}]\label{thm:jones}
Let braids $\beta\in B_{n}$ and $\gamma\in B_{m}$ close up to equivalent oriented
links, with $\gamma$ having the minimal possible number of strands $m$ for this
link type. Then
\[
\bigl|e(\beta)-e(\gamma)\bigr|\ \le\ n-m .
\]
\end{theorem}

By employing \Cref{thm:jones}, one can verify that the exponent sum of a minimal-strand representative is an invariant of $L$, as follows.

\begin{lemma}\label{lem:writhe-invariant}
All minimal-strand closed-braid representatives of a link $L$ have the same
exponent sum. Consequently, the exponent sum $e$ of a minimal-strand
representative is an invariant of $L$.
\end{lemma}

\begin{proof}
Let $n$ be the braid index of $L$, and let $\beta,\beta'\in B_{n}$ be any two
minimal-strand representatives of $L$. By \Cref{thm:jones},
\[
\bigl|e(\beta)-e(\beta')\bigr|\ \le\ n-n=0 ,
\]
hence $e(\beta)=e(\beta')$. Thus the
minimal exponent sum is well-defined and is an invariant of $L$.
\end{proof}

Combining these results, we can establish \Cref{thm:A}.

\begin{proof}[Proof of \Cref{thm:A}]
Let $n=b(K)$, and choose a minimal-strand representative $\beta\in B_n$ of
$K$. By \Cref{lem:writhe-invariant}, the exponent sum of a minimal-strand
representative depends only on the oriented link type, so we may define
$w_{\min}(K):=e(\beta)$.

We first check explicitly that mirroring preserves minimality. The braid
$m(\beta)\in B_n$ represents $mK$, and hence $b(mK)\le n=b(K)$. Applying the
same inequality to $mK$ and using $m(mK)=K$ gives $b(K)\le b(mK)$. Therefore
$b(mK)=b(K)=n$, so $m(\beta)$ is a minimal-strand representative of $mK$, with
$e(m(\beta))=-e(\beta)$. Reversing the orientation changes neither the number
of strands nor the exponent sum. Thus $r(m(\beta))$ is a minimal-strand
representative of $mrK$ and also has exponent sum $-e(\beta)$.

If $K$ is positive amphichiral, then $K\simeq mK$ as oriented knot types;
if $K$ is negative amphichiral, then $K\simeq mrK$. In either case,
\Cref{lem:writhe-invariant} gives $e(\beta)=-e(\beta)$. Consequently,
\begin{equation}\label{eq:wmin-zero}
w_{\min}(K)=e(\beta)=0.
\end{equation}

Write $\beta=\sigma_{i_1}^{\varepsilon_1}\cdots
\sigma_{i_\ell}^{\varepsilon_\ell}$, where $\varepsilon_j\in\{\pm1\}$. Since
$\varepsilon_j\equiv1\pmod2$ for every $j$, we have
$e(\beta)\equiv\ell\pmod2$, and hence $\ell$ is even. On the other hand,
$\pi(\beta)$ is a product of $\ell$ transpositions. Because
$\widehat{\beta}$ is a knot, $\pi(\beta)$ is also an $n$-cycle, and therefore
\[
1=(-1)^\ell=\operatorname{sgn}(\pi(\beta))=(-1)^{n-1}.
\]
Thus $n$ is odd.
\end{proof}

By applying the contrapositive of the theorem, it follows that the $(2,k)$-torus knot is not amphichiral.

\begin{corollary}\label{cor:torus-chiral}
For an odd integer $k \ge 3$, the $(2,k)$-torus knot has braid index $2$ and is
chiral.
\end{corollary}

\begin{proof}
The $(2,k)$-torus knot is naturally represented as the closure of the $2$-braid
$\sigma_1^k\in B_2$. Since the $(2,k)$-torus knot is not the unknot for $k\ge 3$,
its braid index is strictly greater than $1$, and thus exactly $2$. By the
contrapositive of \Cref{thm:A}, any knot of braid index $2$ cannot be
amphichiral; since the $(2,k)$-torus knot has braid index $2$, it follows that
it is chiral.
\end{proof}

\section{Part II: Classification of Prime Amphichiral Knots of Braid Index Three}\label{sec:partII}

\noindent The engine of Part~II is a single concrete observation. Write a braid-index-$3$
alternating knot as the closure of a standard form
$\beta_c=\sigma_1^{a_1}\sigma_2^{-b_1}\cdots\sigma_1^{a_n}\sigma_2^{-b_n}$, and record only
its exponents in a word $c=(a_1,b_1,\dots,a_n,b_n)$. Then the two operations that
decide amphichirality act on $c$ in the simplest possible way: taking the mirror image
rotates $c$ by one entry, and taking the mirror-reverse reflects it
(\Cref{lem:dihedral}). Amphichirality of $\widehat{\beta_c}$ thus becomes a purely
combinatorial question of whether some rotation or reflection returns $c$ to
itself. Reading this condition off yields \Cref{thm:B,thm:C}. The rest of the section
turns this observation into the classification, drawing on Markov's theorem, the Murasugi
normal form, and the Birman--Menasco classification of closed $3$-braids.

\begin{definition}[$M$-equivalence {\cite{KasselTuraev}}]\label{def:Mequiv}
Two braids $\beta\in B_{n}$ and $\beta'\in B_{m}$ (possibly with $n\ne m$) are
\emph{$M$-equivalent} if one can be obtained from the other by a finite sequence
of the following two Markov moves.
\begin{enumerate}[label=\textup{(\alph*)}]
\item \emph{Markov move I (conjugation):} within $B_{n}$,
$\beta\leftrightarrow\alpha\beta\alpha^{-1}$ for $\alpha\in B_{n}$. That is, a braid
is replaced by a conjugate, keeping the number of strands fixed.
\item \emph{Markov move II ((de)stabilization):} for
$\beta\in B_{n}\subset B_{n+1}$, $\beta\leftrightarrow\beta\sigma_{n}^{\pm1}\in
B_{n+1}$. That is, one strand together with a single crossing is added
(stabilization) or removed (destabilization).
\end{enumerate}
\end{definition}

\begin{theorem}[Markov Theorem {\cite[Theorem~2.8]{KasselTuraev}}]\label{thm:markov}
Two braids (possibly with different numbers of strands) have isotopic closures
in $\mathbb{R}^{3}$ if and only if these braids are $M$-equivalent.
\end{theorem}

\noindent We begin the classification of prime amphichiral knots of braid index~$3$ by recalling the Murasugi classification of $3$-braids. This theorem expresses every $3$-braid in a uniform normal form consisting of a power of the full twist and one of three basic types. It will provide the framework for identifying the braid representatives relevant to our classification and for introducing the standard form used throughout the remainder of this section.

\begin{theorem}[Murasugi classification of $3$-braids {\cite[Theorem~6.1]{DasbachLowrance2011}}]\label{thm:murasugi}
Write $\Delta=\sigma_{1}\sigma_{2}\sigma_{1}$, so that
$\Delta^{2}=(\sigma_{1}\sigma_{2}\sigma_{1})^{2}=(\sigma_{1}\sigma_{2})^{3}$. Every
$3$-braid is conjugate to a braid of the form
\[
(\sigma_{1}\sigma_{2})^{3d}\,w=\Delta^{2d}\,w,\qquad d\in\mathbb{Z},
\]
where $w$ is one of the following types:
\begin{enumerate}[label=\textup{(\roman*)}]
\item $\sigma_{1}^{a_1}\sigma_{2}^{-b_1}\cdots\sigma_{1}^{a_\ell}\sigma_{2}^{-b_\ell}$
with $\ell\ge 1$ and all $a_i,b_i>0$;
\item $\sigma_{2}^{k}$ for some $k\in\mathbb{Z}$;
\item $\sigma_{1}^{m}\sigma_{2}^{-1}$ with $m\in\{-1,-2,-3\}$.
\end{enumerate}
Moreover, the exponent $d$ and the type of $w$ are invariants of the conjugacy
class: the representative is unique in Types~\textup{(ii)} and~\textup{(iii)}, and in
Type~\textup{(i)} it is unique up to a cyclic permutation of the syllable pairs
$(a_i,b_i)$.
\end{theorem}

\begin{definition}[Murasugi normal form]\label{def:murasugi_nf}
A $3$-braid is said to be in \emph{Murasugi normal form} if it is written as
\[
\Delta^{2d}w=(\sigma_{1}\sigma_{2})^{3d}w,\qquad d\in\mathbb{Z},
\]
with $w$ of one of Types~\textup{(i)}--\textup{(iii)} of \Cref{thm:murasugi}.
We call $\Delta^{2}$ the \emph{full twist} and the integer $d$ the
\emph{full-twist power} of the braid.
\end{definition}

\noindent The full twist is central in $B_{3}$. Indeed, the braid relation
$\sigma_{1}\sigma_{2}\sigma_{1}=\sigma_{2}\sigma_{1}\sigma_{2}$ gives
\[
\Delta\sigma_{1}\Delta^{-1}=\sigma_{2},\qquad
\Delta\sigma_{2}\Delta^{-1}=\sigma_{1},
\]
so conjugation by the half twist $\Delta$ interchanges the two generators, and
conjugation by $\Delta^{2}$ fixes each of them. Hence $\Delta^{2}$ commutes with all of
$B_{3}$, and in fact $Z(B_{3})=\langle\Delta^{2}\rangle\cong\mathbb{Z}$. See, e.g.,
\cite[Theorem~1.24]{KasselTuraev}. In particular the
factor $\Delta^{2d}$, which represents $d$ full twists, commutes with $w$, so the two
factors in \Cref{def:murasugi_nf} may be interchanged freely. In this terminology,
\Cref{thm:murasugi} says that every $3$-braid is conjugate to one in Murasugi normal
form, and that the full-twist power and the type of $w$ are invariants of the conjugacy
class. Throughout the rest of this section we abbreviate
\[
h:=\Delta^{2}=(\sigma_{1}\sigma_{2})^{3},
\]
so that a Murasugi normal form reads $h^{d}w$.

\begin{lemma}[Twist-free Type-\textup{(i)} forms]\label{lem:alt-standard}
Let
\[
\beta_c=
\sigma_{1}^{a_1}\sigma_{2}^{-b_1}\cdots
\sigma_{1}^{a_n}\sigma_{2}^{-b_n},
\qquad a_i,b_i\ge1.
\]
Then the standard closed-braid diagram $D_c$ of $\widehat{\beta_c}$ is
alternating. Consequently, $\widehat{\beta_c}$ is an alternating link; if the
closure is connected, it is an alternating knot.
\end{lemma}

\begin{proof}
Every $\sigma_1$-crossing of $D_c$ is positive and every $\sigma_2$-crossing
is negative. Inside a syllable $\sigma_i^{\pm k}$, the two participating
strands form a twist region and each strand passes alternately over and under.
At the junction between consecutive $\sigma_1$- and $\sigma_2$-syllables, the
change of generator together with the opposite signs continues the same
over--under alternation. The closing arcs introduce no crossings, so the
pattern continues around every component. Hence $D_c$ is alternating.
\end{proof}

\noindent Building on \Cref{def:murasugi_nf}, we introduce the encoding used
throughout Part~II. It records the twist-free case $d=0$ whose factor $w$ is of
Type~\textup{(i)}. By \Cref{lem:alt-standard}, its standard closure diagram is
alternating, although the closure may have more than one component.

\begin{definition}[Alternating standard form]\label{def:standard_form}
By a \emph{word} of length $2n$ \textup{(}$n\ge1$\textup{)} we mean a finite sequence
\[
c=(c_{1},c_{2},\dots,c_{2n}),\qquad c_{i}\in\mathbb{Z}_{>0},
\]
of strictly positive integers. A word is a sequence, not an equivalence class: two words
are equal precisely when they have the same length and the same entry in every position.
We name the entries in alternating fashion,
\[
a_{i}:=c_{2i-1},\qquad b_{i}:=c_{2i}\qquad(1\le i\le n),
\]
so that $c=(a_{1},b_{1},a_{2},b_{2},\dots,a_{n},b_{n})$, and we call $(a_{i},b_{i})$ the
$i$-th \emph{syllable pair} of $c$. All indices of entries are read modulo $2n$, with
representatives in $\{1,\dots,2n\}$; thus $c_{i+2n}=c_{i}$ by convention. To such a word
we associate the $3$-braid
\[
\beta_c = \sigma_{1}^{a_1}\sigma_{2}^{-b_1}\sigma_{1}^{a_2}\sigma_{2}^{-b_2}\cdots
\sigma_{1}^{a_n}\sigma_{2}^{-b_n},
\]
the alternating product of $n$ positive $\sigma_{1}$-syllables and $n$ negative
$\sigma_{2}$-syllables whose exponents are the entries of $c$: the odd-position entries
$a_{i}$ are the $\sigma_{1}$-exponents and the even-position entries $b_{i}$ are the
$\sigma_{2}$-exponents. We say $\beta\in B_{3}$ is in \emph{alternating standard form} if
$\beta=\beta_c$ for some word $c$. That is, $\beta$ is in Murasugi normal
form (\Cref{def:murasugi_nf}) with $d=0$ and $w$ of Type~(i). Its closure
$\widehat{\beta_c}$ is then an alternating link by \Cref{lem:alt-standard}, and $c$ is
the word of syllable exponents associated with $\beta_c$.
\end{definition}

\noindent Two operations act on the words of \Cref{def:standard_form}. In terms of the
entries $c=(c_{1},\dots,c_{2n})$, with indices read modulo $2n$, they are
\begin{enumerate}[label=\textup{(\alph*)}]
\item the \emph{rotation} $\rho$, given by $\rho(c)_{i}=c_{i+1}$, that is,
\[
\rho\colon\ (a_1,b_1,a_2,b_2,\dots,a_n,b_n)\longmapsto
(b_1,a_2,b_2,\dots,a_n,b_n,a_1);
\]
\item the \emph{reflection} $\tau$, given by $\tau(c)_{i}=c_{2n+1-i}$, that is,
\[
\tau\colon\ (a_1,b_1,a_2,b_2,\dots,a_n,b_n)\longmapsto
(b_n,a_n,\dots,b_1,a_1).
\]
\end{enumerate}
Both preserve the length $2n$ and the positivity of the entries, so both act on the set of
words of length $2n$. Each of them sends odd positions to even positions and vice versa,
hence interchanges the roles of the $\sigma_{1}$-exponents $a_{i}$ and the
$\sigma_{2}$-exponents $b_{i}$. From the displayed formulas one checks
$\rho^{2n}=\tau^{2}=1$ and $\tau\rho\tau^{-1}=\rho^{-1}$, so $\rho$ and $\tau$ generate a
dihedral group $D_{4n}$ of order $4n$ acting on the words of length $2n$.

\smallskip
\noindent Two features of a word, both expressed through this action, drive the
classification.
\begin{itemize}
\item We call $c$ a \emph{palindrome} if there is an integer $s\ge0$ with
\[
(\rho^{s}c)_{i}=(\rho^{s}c)_{2n+1-i}\qquad\text{for all }i\in\{1,\dots,2n\},
\]
that is, if some rotation of $c$ reads the same forwards as backwards.
\item The \emph{period} of $c$ is the least integer $p>0$ with $\rho^{p}(c)=c$.
We say that $c$ \emph{has odd period} when this least period $p$ is odd.
\end{itemize}

\noindent It is convenient to rewrite the palindrome condition as a fixed-point condition
in $D_{4n}$. By the definition of $\tau$, the equality
$(\rho^{s}c)_{i}=(\rho^{s}c)_{2n+1-i}$ for all $i$ says precisely that $\rho^{s}(c)$ is
fixed by $\tau$, that is, $\tau\rho^{s}(c)=\rho^{s}(c)$. Applying $\rho^{-s}$ to both
sides gives $\rho^{-s}\tau\rho^{s}(c)=c$, and $\tau\rho^{s}=\rho^{-s}\tau$ evaluates the
left-hand operator as $\rho^{-s}\tau\rho^{s}=\rho^{-2s}\tau=\tau\rho^{2s}$. Hence, with
the same integer $s$,
\begin{equation}\label{eq:palindrome}
c\ \text{is a palindrome}
\quad\Longleftrightarrow\quad
\tau\rho^{2s}(c)=c\ \text{for some } s\ge0 .
\end{equation}

\begin{lemma}\label{lem:dihedral}
Let $c=(a_1,b_1,a_2,b_2,\dots,a_n,b_n)$ be a word of length $2n$, and set
\[
\beta_c=\sigma_{1}^{a_1}\sigma_{2}^{-b_1}\sigma_{1}^{a_2}\sigma_{2}^{-b_2}\cdots
\sigma_{1}^{a_n}\sigma_{2}^{-b_n}.
\]
Let $\varphi$ be the surjective homomorphism to the four-element operation
group determined by
\[
\varphi\colon D_{4n}\longrightarrow G=\{1,m,r,mr\},
\qquad \varphi(\rho)=m,\qquad \varphi(\tau)=mr,
\]
where $\rho$ denotes the rotation and $\tau$ the reflection defined
above. Then the following hold for every word $c$.
\begin{enumerate}[label=\textup{(\alph*)}]
\item $\widehat{\beta_{\rho(c)}}=m(\widehat{\beta_c})$.
\item $\widehat{\beta_{\tau(c)}}=mr(\widehat{\beta_c})$.
\item $\widehat{\beta_{g(c)}}=\varphi(g)(\widehat{\beta_c})$ for every $g\in D_{4n}$.
\item $\beta_{\rho^{2}(c)}$ is conjugate to $\beta_c$ in $B_3$; hence
    $\beta_{\rho^{2j}(c)}\sim\beta_c$ for every $j\in\mathbb{Z}$. In particular the
    closures agree as oriented knots,
    \[
    \widehat{\beta_c}=\widehat{\beta_{\rho^{2}(c)}}=\widehat{\beta_{\rho^{4}(c)}}
    =\cdots ,
    \]
    that is, rotations by even powers of $\rho$ produce knots that agree even in
    orientation.
\end{enumerate}
\end{lemma}

\begin{proof}
We use three facts. First, by \Cref{thm:markov}, conjugation and Markov moves preserve
the closure. Second, for $\Delta=\sigma_{1}\sigma_{2}\sigma_{1}$ we have
$\Delta\sigma_{1}\Delta^{-1}=\sigma_{2}$ and $\Delta\sigma_{2}\Delta^{-1}=\sigma_{1}$, so
conjugation by $\Delta$ interchanges $\sigma_{1}$ and $\sigma_{2}$ throughout a word.
Third, transferring an initial segment of a word to its rear is a conjugation: if
$\alpha=g\,u$ is any factorization in $B_{3}$, then
\begin{equation}\label{eq:cyclic-conj}
g^{-1}\alpha\,g=g^{-1}(g\,u)\,g=u\,g ,
\end{equation}
so the word $u\,g$, obtained from $g\,u$ by moving the segment $g$ from the front to the
back, is conjugate to $\alpha$ and has the same closure.

\smallskip
\noindent\emph{\textup{(a)}} Since $\widehat{m(\beta_c)}=m(\widehat{\beta_c})$ and the
mirror flips every exponent sign,
\[
m(\beta_c)=\sigma_{1}^{-a_1}\sigma_{2}^{b_1}\sigma_{1}^{-a_2}\sigma_{2}^{b_2}\cdots
\sigma_{1}^{-a_n}\sigma_{2}^{b_n}.
\]
Conjugating by $\Delta$ interchanges $\sigma_{1}$ and $\sigma_{2}$:
\[
\Delta m(\beta_c)\Delta^{-1}
=\sigma_{2}^{-a_1}\sigma_{1}^{b_1}\sigma_{2}^{-a_2}\sigma_{1}^{b_2}\cdots
\sigma_{2}^{-a_n}\sigma_{1}^{b_n}.
\]
This word begins with the syllable $\sigma_{2}^{-a_1}$, so we apply
\eqref{eq:cyclic-conj} with $g=\sigma_{2}^{-a_1}$ and
$u=\sigma_{1}^{b_1}\sigma_{2}^{-a_2}\sigma_{1}^{b_2}\cdots\sigma_{2}^{-a_n}\sigma_{1}^{b_n}$;
conjugating by $g$ carries that leading syllable to the rear and yields
\[
g^{-1}\bigl(\Delta m(\beta_c)\Delta^{-1}\bigr)g
=u\,g
=\sigma_{1}^{b_1}\sigma_{2}^{-a_2}\sigma_{1}^{b_2}\sigma_{2}^{-a_3}\cdots
\sigma_{1}^{b_n}\sigma_{2}^{-a_1}
=\beta_{\rho(c)} ,
\]
the last equality because the syllable pairs of $\rho(c)$ are
$(b_1,a_2),(b_2,a_3),\dots,(b_n,a_1)$. All the braids displayed here are conjugate, so
they have a common closure, and therefore
$\widehat{\beta_{\rho(c)}}=\widehat{m(\beta_c)}=m(\widehat{\beta_c})$.

\smallskip
\noindent\emph{\textup{(b)}} The reverse reads $\beta_c$ backwards and the mirror flips
every exponent sign, so
\[
mr(\beta_c)
=\sigma_{2}^{b_n}\sigma_{1}^{-a_n}\sigma_{2}^{b_{n-1}}\sigma_{1}^{-a_{n-1}}\cdots
\sigma_{2}^{b_1}\sigma_{1}^{-a_1},
\]
and conjugating by $\Delta$ gives
\[
\Delta mr(\beta_c)\Delta^{-1}
=\sigma_{1}^{b_n}\sigma_{2}^{-a_n}\sigma_{1}^{b_{n-1}}\sigma_{2}^{-a_{n-1}}\cdots
\sigma_{1}^{b_1}\sigma_{2}^{-a_1}=\beta_{\tau(c)},
\]
this time with no further rearrangement needed, since the $i$-th syllable pair of
$\tau(c)$ is $(b_{n+1-i},a_{n+1-i})$. Hence
$\widehat{\beta_{\tau(c)}}=mr(\widehat{\beta_c})$.

\smallskip
\noindent\emph{\textup{(c)}} The assignment $\varphi(\rho)=m$, $\varphi(\tau)=mr$
respects the defining relations $\rho^{2n}=\tau^{2}=1$ and $\tau\rho\tau^{-1}=\rho^{-1}$
of $D_{4n}$ (in $G$ one has $m^{2}=(mr)^{2}=1$ and $(mr)m(mr)^{-1}=m$), so $\varphi$ is a
well-defined homomorphism, as asserted in the statement. Writing $g\in D_{4n}$ as a word
in $\rho$ and $\tau$ and applying (a) and (b) repeatedly gives
$\widehat{\beta_{g(c)}}=\varphi(g)(\widehat{\beta_c})$ for every $g\in D_{4n}$.

\smallskip
\noindent\emph{\textup{(d)}} Here the leading \emph{syllable pair} is transferred to the
rear: applying \eqref{eq:cyclic-conj} with $g=\sigma_{1}^{a_1}\sigma_{2}^{-b_1}$ and
$u=\sigma_{1}^{a_2}\sigma_{2}^{-b_2}\cdots\sigma_{1}^{a_n}\sigma_{2}^{-b_n}$, so that
$\beta_c=g\,u$, we obtain
\[
g^{-1}\beta_c\,g
=u\,g
=\sigma_{1}^{a_2}\sigma_{2}^{-b_2}\cdots\sigma_{1}^{a_n}\sigma_{2}^{-b_n}\,
 \sigma_{1}^{a_1}\sigma_{2}^{-b_1}
=\beta_{\rho^{2}(c)},
\]
since $\rho^{2}$ shifts the syllable pairs of $c$ cyclically by one. Hence
$\beta_{\rho^{2}(c)}\sim\beta_c$ in $B_3$, and iterating gives
$\beta_{\rho^{2j}(c)}\sim\beta_c$ for every $j\in\mathbb{Z}$. Conjugate braids have
isotopic closures as oriented links, so
$\widehat{\beta_{\rho^{2j}(c)}}=\widehat{\beta_c}$ for all $j$, which is the displayed
chain of equalities. This is consistent with~(c): one has $\varphi(\rho^{2})=m^{2}=1$, so
$\rho^{2}$ acts trivially on the closure.
\end{proof}

\begin{theorem}[Conjugacy of alternating $3$-braids]\label{thm:conjugacy_3braids}
Let $\beta_c$ and $\beta_{c'}$ be $3$-braids in the alternating standard form of
\Cref{def:standard_form}, with words
$c=(a_1,b_1,\dots,a_n,b_n)$ and $c'=(a_1',b_1',\dots,a_{n'}',b_{n'}')$
of lengths $2n$ and $2n'$. Then $\beta_c$ and $\beta_{c'}$ are conjugate in $B_3$ if and
only if $n=n'$, $c'=\rho^{2j}(c)$ for some integer $j$.
\end{theorem}

\begin{proof}
If $n=n'$ and $c'=\rho^{2j}(c)$, then $\beta_c\sim\beta_{c'}$ in $B_3$ by
\Cref{lem:dihedral}(d).

Conversely, suppose $\beta_c\sim\beta_{c'}$. Both are Murasugi normal forms of
full-twist power $0$ with factor of Type~(i) (\Cref{def:standard_form}), so
they represent the same conjugacy class. By
\Cref{thm:murasugi}, a conjugacy class of $B_3$ determines its
Murasugi normal form uniquely, the only freedom for a Type-(i) factor being the
cyclic order of the syllable pairs; by \Cref{lem:dihedral}(d) this freedom is
exactly the pair-rotation $\rho^{2}$. Hence $c$ and $c'$ lie in a common
$\rho^{2}$-orbit, i.e.\ $n=n'$ and $c'=\rho^{2j}(c)$ for some $j$.
\end{proof}

\noindent The passage from conjugacy classes in $B_3$ to oriented link types, for which
a single link type may be carried by two distinct conjugacy classes, is the
contribution of Birman and Menasco~\cite{BM1993}, recalled in
\Cref{thm:BM} below. Before classifying the amphichiral knots we recall the Birman--Menasco
apparatus for closed $3$-braids: \emph{flypes}, the conjugacy-class classification, and the
invertibility criterion. These are stated here so that the results below, particularly
\Cref{thm:amphi-implies-alt} and the flype case of \Cref{thm:C}, may refer
to them without forward references.

\begin{definition}[Flype {\cite{BM2008}}]\label{def:flype}
\leavevmode
\begin{enumerate}[label=\textup{(\alph*)}]
\item Let $D$ be a link diagram containing a $2$-string tangle $T$ (the part of
$D$ lying in a disk that meets the rest of the diagram in exactly four boundary
points) together with a single crossing $c$ adjacent to two of those four points.
A \emph{flype} is the move that rotates the tangle $T$ through $180^{\circ}$ about
an axis lying in the plane of the diagram, thereby carrying the crossing $c$ from
one side of $T$ to the other. A flype takes a diagram to a diagram of the same
link. 

\item In the setting of closed $3$-braids, a flype is applied to a braid of the
form
\[
\beta=\sigma_{1}^{p}\sigma_{2}^{q}\sigma_{1}^{r}\sigma_{2}^{\varepsilon},
\qquad \varepsilon=\pm1,
\]
in which $\sigma_{2}^{\varepsilon}$ is the flyped crossing. The literal
Birman--Menasco flype partner is
\[
\beta^{\mathrm f}
=\sigma_{1}^{r}\sigma_{2}^{q}\sigma_{1}^{p}\sigma_{2}^{\varepsilon}.
\]
At the level of conjugacy classes we shall also use the cyclically conjugate
representative
\[
\widetilde\beta^{\mathrm f}
=\sigma_{1}^{p}\sigma_{2}^{\varepsilon}\sigma_{1}^{r}\sigma_{2}^{q}
\sim\beta^{\mathrm f}.
\]
Since a flype is a move between diagrams of the same link,
$\widehat{\beta}=\widehat{\beta^{\mathrm f}}=
\widehat{\widetilde\beta^{\mathrm f}}$, although the corresponding braids need
not be conjugate to $\beta$ in $B_3$.
A link $L$ is said to \emph{admit a flype} if its conjugacy class of $3$-braids
has a representative of the form
$\sigma_{1}^{p}\sigma_{2}^{q}\sigma_{1}^{r}\sigma_{2}^{\varepsilon}$.
\item A flype is \emph{positive} or \emph{negative} according as
$\varepsilon=+1$ or $\varepsilon=-1$.
\item A flype is \emph{non-degenerate} when $\beta$ and its literal flype
partner $\beta^{\mathrm f}$
lie in distinct conjugacy classes.
\end{enumerate}
\end{definition}

The Birman--Menasco classification of closed $3$-braids (\Cref{thm:BM}) is the
basis for the classification of amphichiral knots carried out below. We also
record their invertibility criterion (\Cref{thm:invert}), which is needed to
establish full amphichirality in the flype case.

\begin{theorem}[Birman--Menasco {\cite[The Classification Theorem (Version~1)]{BM1993}}; see also {\cite[Theorem~5]{KoLee1999}}]
\label{thm:BM}
A link $L$ which is represented by a closed $3$-braid admits a unique conjugacy
class of $3$-braid representatives, with the following exceptions.
\begin{enumerate}[label=\textup{(\roman*)}]
\item $L$ is the unknot, which has three conjugacy classes of $3$-braid
representatives, namely the classes of $\sigma_{1}\sigma_{2}$, $\sigma_{1}\sigma_{2}^{-1}$, and
$\sigma_{1}^{-1}\sigma_{2}^{-1}$.
\item $L$ is a type $(2,k)$-torus link with $|k|\ge 2$, which has two conjugacy
classes of $3$-braid representatives, namely the classes of $\sigma_{1}^{k}\sigma_{2}$ and
$\sigma_{1}^{k}\sigma_{2}^{-1}$.
\item $L$ is one of a special class of links of braid index $3$ which have exactly two conjugacy classes of $3$-braid representatives, namely the classes of
\[
\sigma_{1}^{p}\sigma_{2}^{q}\sigma_{1}^{r}\sigma_{2}^{\varepsilon}
\qquad\text{and}\qquad
\sigma_{1}^{p}\sigma_{2}^{\varepsilon}\sigma_{1}^{r}\sigma_{2}^{q},
\]
where $\varepsilon=\pm1$, $|q|\ge 2$, the integers $p,q+\varepsilon,r$ are
distinct, and neither $p$ nor $r$ lies in $\{0,\varepsilon,2\varepsilon\}$.
\end{enumerate}
\end{theorem}

\noindent
As a first concrete illustration of the non-degenerate flype case,
\Cref{fig:braid63closed} shows the knot $6_3$, the closure of the
alternating standard form $\beta_c$ with the palindromic word
$c=(2,1,1,2)$. It is the smallest amphichiral knot of braid index~$3$
belonging to the non-degenerate flype family $(k,1,1,k)$, and we include
it here as the first example of that family.

\begin{figure}[ht!]
    \centering
    \includegraphics[width=0.7\linewidth]{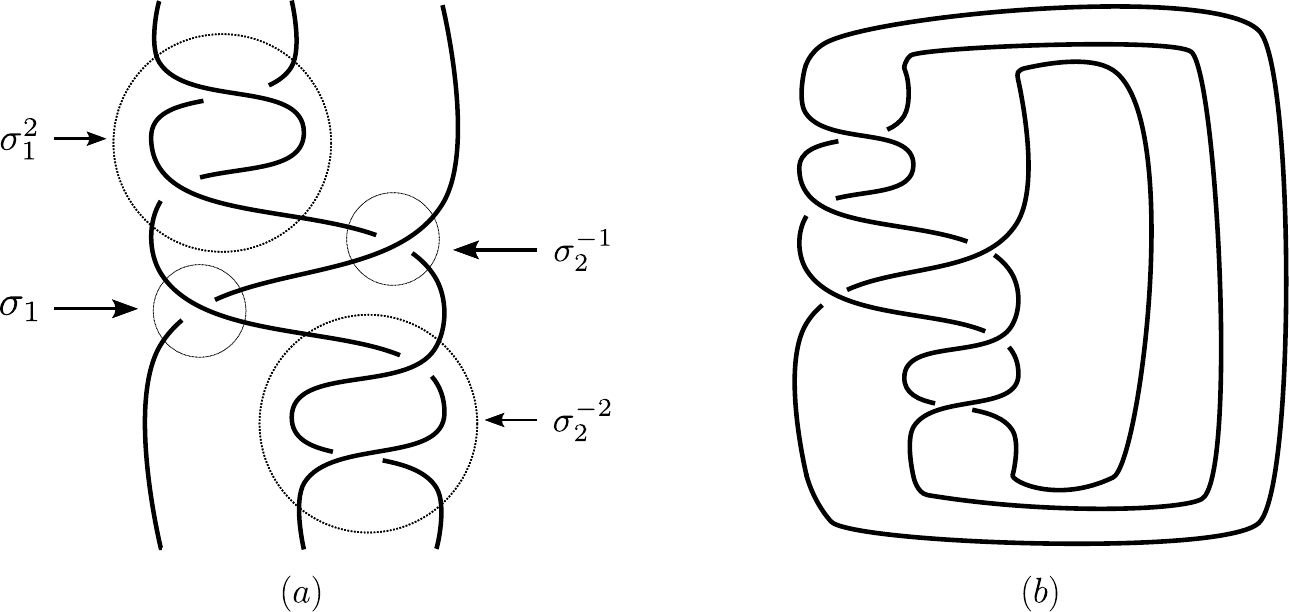}
  \caption{~\textup{(a)} represents the braid
  $\beta_c=\sigma_1^{2}\sigma_2^{-1}\sigma_1\sigma_2^{-2}$, and ~\textup{(b)} represents its closure $\widehat{\beta_c}=6a_3$. The knot is fully amphichiral; here
  $c=(2,1,1,2)$ is a palindrome in the flype family $(k,1,1,k)$ with $k=2$.}
\label{fig:braid63closed}
\end{figure}

\begin{theorem}[The invertibility theorem {\cite[Theorem~2]{BM2008}}]\label{thm:invert}
Let $L$ be a link of braid index $3$ with oriented $3$-braid representative
$\beta$. Then $L$ is non-invertible if and only if $\beta$ and $r(\beta)$ lie
in distinct conjugacy classes \emph{and} the class of $\beta$ contains no
representative admitting a non-degenerate flype.
\end{theorem}

\noindent The following lemma is used in the proof of \Cref{thm:amphi-implies-alt}
below, where it supplies the case of a non-degenerate flype (Case~B of that proof):
such a flype leaves the full-twist power unchanged, because the flype partner is, up
to conjugacy, the reverse braid.
\begin{lemma}[Flypes and the full-twist power]\label{lem:flype-twist}
\leavevmode
\begin{enumerate}[label=\textup{(\alph*)}]
\item For every $\beta\in B_3$, the reverse $r(\beta)$ has the same full-twist power in
Murasugi normal form as $\beta$.
\item Let $\beta=\sigma_1^{p}\sigma_2^{q}\sigma_1^{r}\sigma_2^{\varepsilon}$ admit a
non-degenerate flype, and let
$\widetilde\beta^{\mathrm f}=
\sigma_1^{p}\sigma_2^{\varepsilon}\sigma_1^{r}\sigma_2^{q}$ be the conjugate
flype-partner representative from \Cref{def:flype}. Then
$\widetilde\beta^{\mathrm f}$ is conjugate to $r(\beta)$. Consequently,
$\beta$ and $\widetilde\beta^{\mathrm f}$ have the same full-twist power.
\end{enumerate}
\end{lemma}

\begin{proof}
\noindent\emph{\textup{(a)}} Reverse is the \emph{anti-homomorphism} of $B_3$ that reads a
word backwards; that is, $r(\alpha\beta)=r(\beta)\,r(\alpha)$ for all
$\alpha,\beta\in B_3$. Since $r(\Delta)=\Delta$, the full twist satisfies
$r(h)=r(\Delta^{2})=\Delta^{2}=h$; as $h$ is central,
\[
r(h^{d}w)=r(w)\,r(h^{d})=r(w)\,h^{d}=h^{d}\,r(w).
\]
Thus, writing $\beta$ in Murasugi normal form $h^{d}w$, we obtain
$r(\beta)=h^{d}\,r(w)$. For each type we show that $r(w)$ is conjugate to a twist-free Murasugi form
of the same type.

\smallskip
\noindent\emph{Type \textup{(i)}.} Here $w=\sigma_1^{a_1}\sigma_2^{-b_1}\cdots
\sigma_1^{a_\ell}\sigma_2^{-b_\ell}=\beta_c$, with word
$c=(a_1,b_1,\dots,a_\ell,b_\ell)$. By the definition of the reverse, $r$ reads the
word backwards, so
\[
r(w)=\sigma_2^{-b_\ell}\sigma_1^{a_\ell}\sigma_2^{-b_{\ell-1}}\sigma_1^{a_{\ell-1}}
\cdots\sigma_2^{-b_1}\sigma_1^{a_1}.
\]
By \eqref{eq:cyclic-conj}, conjugating $r(w)$ by $g=\sigma_2^{-b_\ell}$ moves the leading
syllable to the rear,
\[
g^{-1}\,r(w)\,g
=\sigma_2^{\,b_\ell}\,r(w)\,\sigma_2^{-b_\ell}
=\sigma_1^{a_\ell}\sigma_2^{-b_{\ell-1}}\sigma_1^{a_{\ell-1}}\cdots
\sigma_1^{a_1}\sigma_2^{-b_\ell}
=\beta_{\rho\tau(c)},
\]
a Type-(i) standard form. Hence $r(w)\sim\beta_{\rho\tau(c)}$; in particular $r(w)$
is conjugate to a twist-free Type-(i) form.

\smallskip
\noindent\emph{Type \textup{(ii)}.} Here $r(\sigma_2^{k})=\sigma_2^{k}$, already of
Type~(ii).

\smallskip
\noindent\emph{Type \textup{(iii)}.} Here $r(\sigma_1^{m}\sigma_2^{-1})
=\sigma_2^{-1}\sigma_1^{m}$. Moving the leading syllable $\sigma_2^{-1}$ to the rear by
conjugating with $\sigma_2^{-1}$ as in \eqref{eq:cyclic-conj} gives
\[
\sigma_2\,(\sigma_2^{-1}\sigma_1^{m})\,\sigma_2^{-1}=\sigma_1^{m}\sigma_2^{-1},
\]
so $r(\sigma_1^{m}\sigma_2^{-1})\sim\sigma_1^{m}\sigma_2^{-1}$, again of Type~(iii).

\smallskip
\noindent In every case $r(\beta)$ is conjugate to a Murasugi normal form of power $d$. By
the uniqueness in \Cref{thm:murasugi}, this is its normal form, so its full-twist power
is $d$.

\smallskip
\noindent\emph{\textup{(b)}} Moving the trailing syllable of $r(\beta)$ to the front,
again a conjugation as in \eqref{eq:cyclic-conj},
\[
r(\beta)=\sigma_2^{\varepsilon}\sigma_1^{r}\sigma_2^{q}\sigma_1^{p}
\ \sim\ \sigma_1^{p}\sigma_2^{\varepsilon}\sigma_1^{r}\sigma_2^{q}
=\widetilde\beta^{\mathrm f}.
\]
The full-twist power is a conjugacy invariant and is preserved by reverse by part~(a);
hence $\beta$ and $\widetilde\beta^{\mathrm f}$ share it.
\end{proof}

\noindent We now show that no generality is lost by restricting to alternating standard
forms: amphichirality \emph{alone} forces a knot of braid index~$3$ to be alternating.

\begin{theorem}[Amphichirality and alternating standard form]
\label{thm:amphi-implies-alt}
Let $K$ be a knot of braid index exactly~$3$. If $K$ is amphichiral,
then $K$ admits a minimal $3$-braid representative in alternating
standard form. In particular, $K$ is alternating.
\end{theorem}

\begin{proof}
Represent $K$ as a closed $3$-braid and put a representative in Murasugi normal form,
\[
\beta=(\sigma_1\sigma_2)^{3d}\,w=h^{d}w,\qquad h=(\sigma_1\sigma_2)^{3},\ d\in\mathbb{Z},
\]
with $w$ of Type~(i),~(ii), or~(iii). Since $K$ has braid index $3$, $\beta$ is a
minimal-strand representative, so by \Cref{lem:writhe-invariant} its
exponent sum equals the link invariant $w_{\min}(K)$, and amphichirality gives
$e(\beta)=w_{\min}(K)=0$ by~\eqref{eq:wmin-zero}.

\smallskip
\noindent\emph{Step 1: $w$ is of Type~\textup{(i)}.}
Since $\pi(\sigma_1\sigma_2)=(1\,3\,2)$ has order $3$, we have $\pi(h)=1$ and thus
$\pi(\beta)=\pi(w)$. Now suppose, for the sake of contradiction, that $w$ is not of Type~(i). By
\Cref{thm:murasugi} it is then of Type~(ii) or of Type~(iii), and we treat these two
possibilities in turn. Suppose first that $w$ is of Type~(ii), say $w=\sigma_2^{k}$ for some
$k\in\mathbb{Z}$. Then $\pi(\beta)=\pi(w)=(2\,3)^{k}$, a permutation that fixes the
first strand and hence is never a $3$-cycle, whatever the value of $k$. Consequently
$\widehat{\beta}$ has at least two components, contradicting the hypothesis that
$\widehat{\beta}$ is the knot $K$.
 
Suppose instead that $w$ is of Type~(iii), say $w=\sigma_1^{m}\sigma_2^{-1}$ with
$m\in\{-1,-2,-3\}$. Then $\pi(\beta)=\pi(w)=(1\,2)^{m}(2\,3)$, which is a $3$-cycle
precisely when $m$ is odd. For the even value $m=-2$ the permutation again fixes the
first strand, so $\widehat{\beta}$ is a link of more than one component and the same
contradiction arises. There remain the two odd values $m\in\{-1,-3\}$, and these are
excluded not by a count of components but by the exponent sum. Indeed
$e(w)=e(\sigma_1^{m}\sigma_2^{-1})=m-1$, so that $e(\beta)=6d+m-1$, while
amphichirality forces $e(\beta)=0$ by~\eqref{eq:wmin-zero}. Hence $6d=1-m$, that is,
$6d=2$ when $m=-1$ and $6d=4$ when $m=-3$. Neither $2$ nor $4$ is divisible by $6$, so
no integer $d$ satisfies either equation, and both values of $m$ are impossible.
 
In every case we have reached a contradiction, so $w$ is of Type~(i), say
\[
\beta=h^{d}\,\sigma_1^{a_1}\sigma_2^{-b_1}\cdots\sigma_1^{a_n}\sigma_2^{-b_n},
\qquad a_i,b_i\ge1 .
\]
 
\smallskip
\noindent\emph{Step 2: the mirror carries $d$ to $-d$.}
The mirror flips every exponent sign, and conjugation by
$\Delta=\sigma_1\sigma_2\sigma_1$ interchanges $\sigma_1\leftrightarrow\sigma_2$ while
fixing the central element $h$, as seen in the proof of \Cref{lem:dihedral}(a).
Hence, moving the leading syllable to the rear as in~\eqref{eq:cyclic-conj},
\[
\Delta\,m(\beta)\,\Delta^{-1}
=h^{-d}\,\sigma_2^{-a_1}\sigma_1^{b_1}\cdots\sigma_2^{-a_n}\sigma_1^{b_n}
\ \sim\ h^{-d}\,\sigma_1^{b_1}\sigma_2^{-a_2}\sigma_1^{b_2}\sigma_2^{-a_3}\cdots
\sigma_1^{b_n}\sigma_2^{-a_1},
\]
again a Murasugi normal form of Type~(i), now with full-twist power $-d$, whose closure is
$m(K)$. Since the reverse $r$ preserves the full-twist power by
\Cref{lem:flype-twist}(a), the knot $mr(K)$ likewise has a Type-(i) representative
of power $-d$.

\smallskip
\noindent\emph{Step 3: $d=0$.}
Recall the full-twist power is a conjugacy invariant by \Cref{thm:murasugi}. Amphichirality gives $m(K)=K$ or $mr(K)=K$, so by Step~2 the knot $K$ admits a Type-(i) representative of power $-d$ alongside $\beta$ of power $d$. As $K$ has braid index $3$, it is neither the unknot nor a $(2,k)$-torus knot. Hence by Birman--Menasco (\Cref{thm:BM}) $K$ is represented either by a single conjugacy class of $3$-braids or by two classes related by a non-degenerate flype.

If the class of $\beta$ is the unique conjugacy class of $3$-braids representing $K$,
then the two representatives are conjugate, and comparing their full-twist powers gives
$-d=d$. Otherwise, the class of $\beta$ admits a non-degenerate flype, so that $K$ is
represented by exactly two conjugacy classes; these two classes share the same full-twist
power by \Cref{lem:flype-twist}(b), and the power-$(-d)$ representative lies in one of
them, so again $-d=d$.

\smallskip
\noindent In either case $d=0$. Hence
$\beta=\sigma_1^{a_1}\sigma_2^{-b_1}\cdots\sigma_1^{a_n}\sigma_2^{-b_n}$ is an alternating
standard form (\Cref{def:standard_form}), whose closure is alternating by
\Cref{lem:alt-standard}. Hence $K$ is alternating.
\end{proof}

\noindent We first treat the case in which $K$ is represented by a unique
conjugacy class of $3$-braids. The mechanism is that, under this hypothesis, the
symmetry group $\mathrm{Sym}(K)$ is identified with the stabilizer of $c$ under
the dihedral action; the proof below makes this identification precise. This
settles one of the two halves of the classification: the theorem below is
precisely case~\textup{(A)} of \Cref{thm:C}, and supplies the corresponding half
of \Cref{thm:B}, the remaining flype case being deferred to
\Cref{lem:flype-family,thm:flype-amphi}.

\begin{theorem}[Classification: unique conjugacy class]\label{thm:unique-amphi}
Suppose that
\begin{enumerate}[label=\textup{(\alph*)}]
\item
$\beta_c=
\sigma_{1}^{a_1}\sigma_{2}^{-b_1}\cdots
\sigma_{1}^{a_n}\sigma_{2}^{-b_n}$
is in alternating standard form with all $a_i,b_i\ge1$, and its
closure $K=\widehat{\beta_c}$ is a prime knot of braid index~$3$;
\item $K$ is represented by a unique conjugacy class of $3$-braids; equivalently,
the conjugacy class of $\beta_c$ admits no non-degenerate flype.
\end{enumerate}
Then $K$ is amphichiral if and only if $c$ is a palindrome or has odd period.
Moreover, if $K$ is amphichiral, then:
\begin{enumerate}[label=\textup{(\roman*)}]
  \item $K$ is fully amphichiral $\iff$ $c$ is a palindrome and has odd period;
  \item $K$ is positive amphichiral but not negative amphichiral $\iff$ $c$ has odd period but is not a palindrome;
  \item $K$ is negative amphichiral but not positive amphichiral $\iff$ $c$ is a palindrome but does not have odd period.
\end{enumerate}
\end{theorem}

\begin{example}[The three symmetry types]\label{ex:three-types}
Each of the cases \textup{(i)}--\textup{(iii)} of \Cref{thm:unique-amphi} is
realized by a knot of small crossing number, and the corresponding closed braids
are drawn in \Cref{fig:three-examples}.
\begin{enumerate}[label=\textup{(\roman*)}]
\item The figure-eight knot $4a_{1}=\widehat{\beta_c}$ with $c=(1,1,1,1)$ is a
palindrome of period $p=1$, so both $m$ and $mr$ lie in $\mathrm{Sym}(K)$ and $K$
is fully amphichiral. It is the smallest amphichiral knot of braid index~$3$.
\item The knot $14a_{18676}$, with $c=(1,2,4,1,2,4)$, has odd period $p=3$ but is
not a palindrome, so $m\in\mathrm{Sym}(K)$ while $mr\notin\mathrm{Sym}(K)$ and $K$
is positive amphichiral. Among all prime amphichiral knots of braid index~$3$ with
crossing number at most~$14$ the entry \texttt{14a\_18676} is the only knot that
is positive amphichiral but not fully amphichiral \textup{(}\Cref{tab:bi3-cn14}\textup{)}, hence the sole witness for this case in
that range.
\item The knot $8a_{17}$, with $c=(2,1,1,1,1,2)$, is a palindrome but has period
$6$, so $mr\in\mathrm{Sym}(K)$ while $m\notin\mathrm{Sym}(K)$ and $K$ is negative
amphichiral.
\end{enumerate}
\end{example}

\begin{figure}[p]
\centering
\begin{subfigure}{\linewidth}
  \centering
  \includegraphics[width=0.62\linewidth]{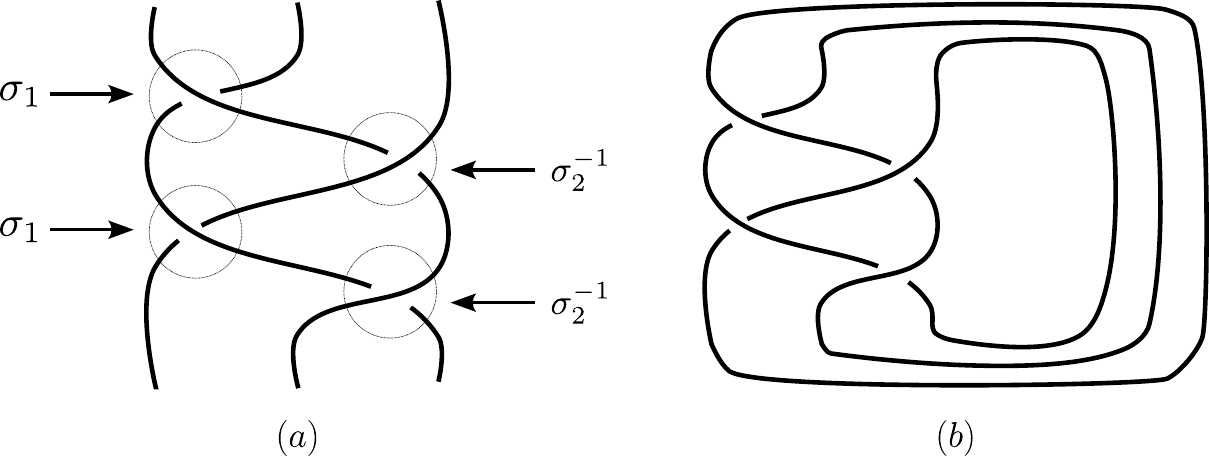}
  \caption{$\beta_c=\sigma_1\sigma_2^{-1}\sigma_1\sigma_2^{-1}$ and its closure
  $\widehat{\beta_c}=4a_1$, the figure-eight knot: $c=(1,1,1,1)$ is a palindrome
  of period $p=1$, so $K$ is fully amphichiral.}
  \label{fig:braid4a1closed}
\end{subfigure}

\vspace{1.2em}

\begin{subfigure}{\linewidth}
  \centering
  \includegraphics[width=0.62\linewidth]{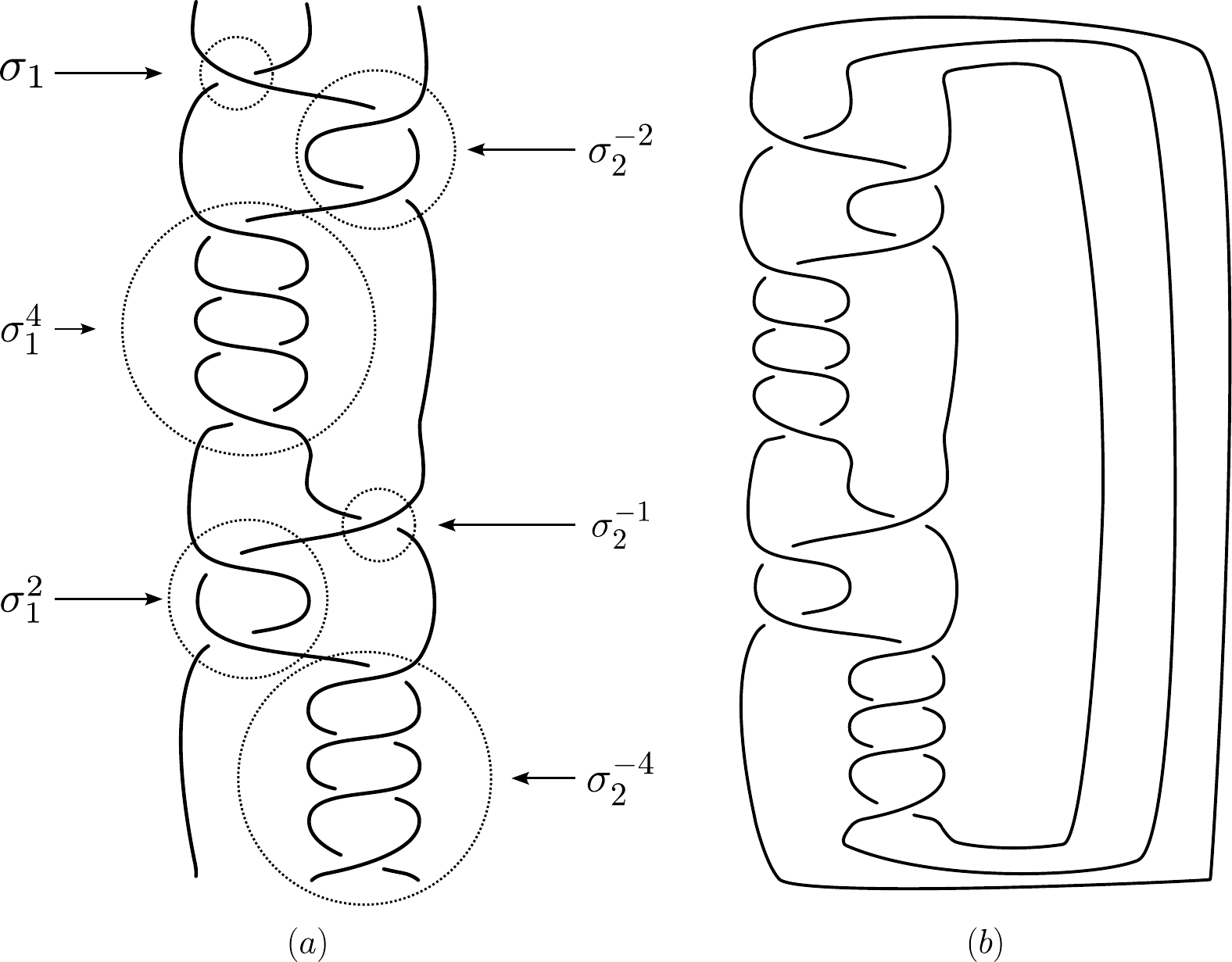}
  \caption{$\beta_c=\sigma_1\sigma_2^{-2}\sigma_1^{4}\sigma_2^{-1}\sigma_1^{2}
  \sigma_2^{-4}$ and its closure $\widehat{\beta_c}=14a_{18676}$: $c=(1,2,4,1,2,4)$
  has odd period $p=3$ but is not a palindrome, so $K$ is positive amphichiral but not negative amphichiral.}
  \label{fig:braid14a18676closed}
\end{subfigure}

\vspace{1.2em}

\begin{subfigure}{\linewidth}
  \centering
  \includegraphics[width=0.62\linewidth]{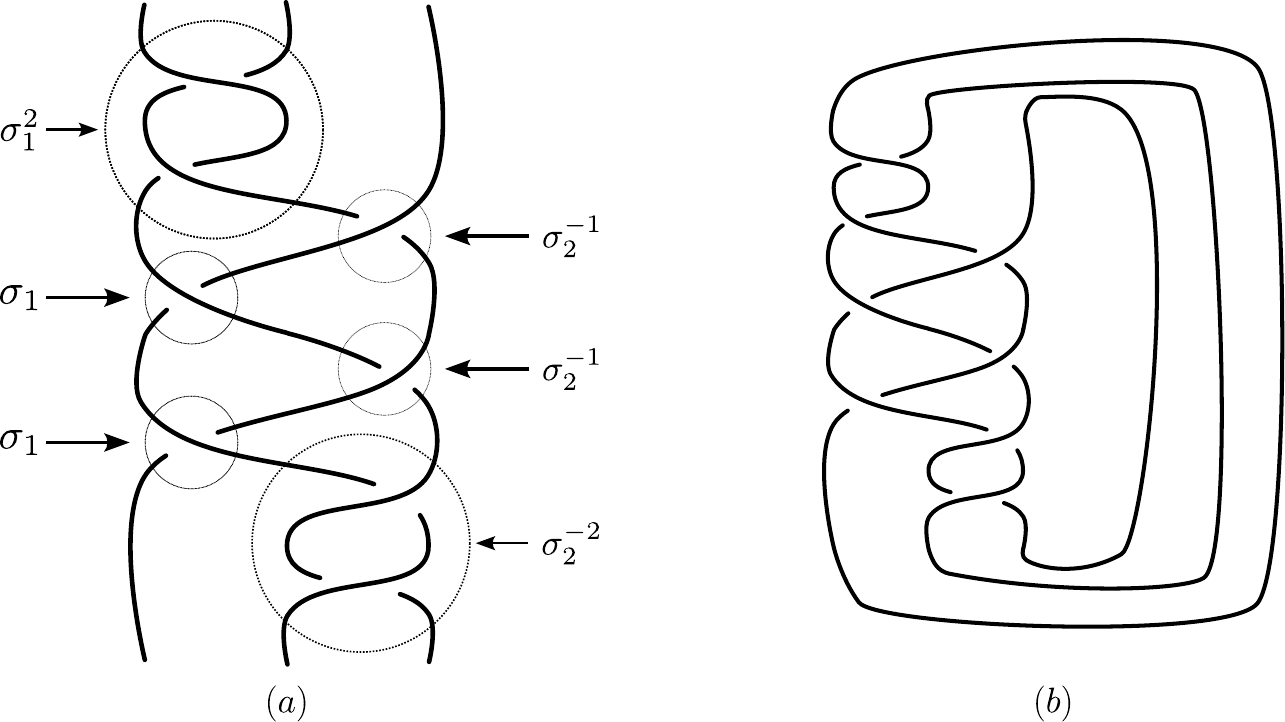}
  \caption{$\beta_c=\sigma_1^{2}\sigma_2^{-1}\sigma_1\sigma_2^{-1}\sigma_1
  \sigma_2^{-2}$ and its closure $\widehat{\beta_c}=8a_{17}$: $c=(2,1,1,1,1,2)$ is
  a palindrome of even period $6$, so $K$ is negative amphichiral but not positive amphichiral.}
  \label{fig:braid8a17closed}
\end{subfigure}
\caption{The three symmetry types of \Cref{thm:unique-amphi}
\textup{(}\Cref{ex:three-types}\textup{)}. In each panel the left-hand diagram is
the braid $\beta_c$ in alternating standard form and the right-hand diagram is its
closure.}
\label{fig:three-examples}
\end{figure}

\begin{proof}
\noindent The criterion is read off from a group-theoretic identification which we
establish first: the homomorphism $\varphi$ of \Cref{lem:dihedral} has kernel
$\langle\rho^{2}\rangle$ and hence induces an isomorphism
$D_{4n}/\langle\rho^{2}\rangle\cong G$, through which $G$ acts on the conjugacy classes
$[c]=\{\rho^{2j}(c):j\in\mathbb{Z}\}$ of standard forms by $\bar g\cdot[c]=[g(c)]$, and
under this action $\mathrm{Sym}(K)$ is the stabilizer of $[c]$.

\smallskip
\noindent\emph{The dihedral criterion.}
Let $\varphi\colon D_{4n}\to G=\{1,m,r,mr\}$ be the homomorphism of
\Cref{lem:dihedral}, with $\varphi(\rho)=m$ and $\varphi(\tau)=mr$, so that
$\widehat{\beta_{g(c)}}=\varphi(g)(K)$ for every $g\in D_{4n}$. We use throughout the
uniqueness hypothesis~(b) in the following form: if a standard form $\beta_{c'}$ has the
same closure as $\beta_c$, then $\beta_{c'}$ and $\beta_c$ are conjugate (otherwise two
distinct conjugacy classes would represent $K$), so by
\Cref{thm:conjugacy_3braids} $c'=\rho^{2j}(c)$ for some integer $j$. Consequently,
for every $g\in D_{4n}$,
\[
\varphi(g)\in\mathrm{Sym}(K)
\iff \widehat{\beta_{g(c)}}=\widehat{\beta_c}
\iff g(c)=\rho^{2j}(c)\ \text{for some }j .
\tag{$\ast$}
\]

The right-hand condition in $(\ast)$ depends on $g$ only through its image in
the quotient $D_{4n}/\langle\rho^{2}\rangle$, and we now make this quotient precise. The
subgroup $\langle\rho^{2}\rangle$ generated by $\rho^{2}$ is normal in $D_{4n}$:
conjugating $\rho^{2k}$ by the generator $\tau$ gives
$\tau\rho^{2k}\tau^{-1}=\rho^{-2k}\in\langle\rho^{2}\rangle$, and conjugation by any power
of $\rho$ fixes it, so the subgroup is stable under all of $D_{4n}$. Moreover
$\langle\rho^{2}\rangle\subseteq\ker\varphi$, since $\varphi(\rho^{2})=m^{2}=1$. To see
that this inclusion is an equality we compare orders. The rotation $\rho$ has order $2n$,
so $\rho^{2}$ has order $n$ and $|\langle\rho^{2}\rangle|=n$; on the other hand $\varphi$
is onto $G$, which has order $4$, so $|\ker\varphi|=|D_{4n}|/|G|=4n/4=n$. As
$\langle\rho^{2}\rangle$ is a subgroup of $\ker\varphi$ of the same finite order, the two
coincide:
\[
\ker\varphi=\langle\rho^{2}\rangle .
\]
By the first isomorphism theorem $\varphi$ therefore descends to an isomorphism
\[
\overline{\varphi}\colon\ D_{4n}/\langle\rho^{2}\rangle\ \xrightarrow{\ \sim\ }\ G,
\qquad
\overline{\varphi}(\bar\rho)=m,\quad
\overline{\varphi}(\bar\tau)=mr,\quad
\overline{\varphi}(\bar\rho\bar\tau)=r,
\]
and a count of orders confirms the match, $|D_{4n}/\langle\rho^{2}\rangle|=4n/n=4=|G|$.
Since $\langle\rho^{2}\rangle$ is normal, the quotient $G\cong
D_{4n}/\langle\rho^{2}\rangle$ acts on the conjugacy classes of standard forms by the
well-defined rule $\bar g\cdot[c]=[g(c)]$. These conjugacy classes are precisely the
$\langle\rho^{2}\rangle$-orbits $[c]=\{\rho^{2j}(c):j\in\mathbb{Z}\}$ by
\Cref{thm:conjugacy_3braids}. Reading $(\ast)$ through this identification,
$\varphi(g)\in\mathrm{Sym}(K)$ holds precisely when $\bar g$ fixes the class $[c]$, and hence
\[
\mathrm{Sym}(K)=\mathrm{Stab}_{G}([c]),
\]
the stabilizer of the class of $c$ in $G$ for the induced action. Reading off the three nontrivial elements via
$(\ast)$, and using $\tau\rho^{k}=\rho^{-k}\tau$ throughout:
\begin{enumerate}[label=\textup{(\alph*)}]
\item $m\in\mathrm{Sym}(K)\iff\rho(c)=\rho^{2j}(c)$ for some $j\iff\rho^{1-2j}(c)=c$ for
some $j$; as $1-2j$ ranges over all odd integers, this says that $\rho^{p}(c)=c$ for some
odd $p$, i.e.\ that $c$ has odd period.
\item $mr\in\mathrm{Sym}(K)\iff\tau(c)=\rho^{2j}(c)$ for some $j$, that is,
$\rho^{-2j}\tau(c)=\tau\rho^{2j}(c)=c$ for some $j$; since $\rho^{2n}=1$ we
may take $j\ge0$, so by~\eqref{eq:palindrome} this says exactly that $c$ is a palindrome.
\item $r\in\mathrm{Sym}(K)\iff\rho\tau(c)=\rho^{2j}(c)$ for some $j$, that is,
$\rho^{1-2j}\tau(c)=\tau\rho^{2j-1}(c)=c$ for some $j$: the same kind of
condition as in~(b), but with an odd power of $\rho$ in place of the even power.
\end{enumerate}
\noindent We remark that condition~(c) is not used directly below; it is subsumed by the relation $r=m\cdot mr$ used in the symmetry-type analysis that follows.

\smallskip
\noindent\emph{Symmetry type.} Assume $K$ is amphichiral. Since $\mathrm{Sym}(K)$ is a
subgroup of $G\cong(\mathbb{Z}/2)^{2}$ and $r=m\cdot mr$, closure gives one implication:
if \emph{both} $m$ and $mr$ lie in $\mathrm{Sym}(K)$, then so does $r$. Here, however, $K$ is amphichiral, so by
\Cref{def:amphi} at least one of $m,mr$ already lies in $\mathrm{Sym}(K)$.
Combined with $r\in\mathrm{Sym}(K)$ this forces the other one in as well (if
$r,m\in\mathrm{Sym}(K)$ then $mr=r\cdot m\in\mathrm{Sym}(K)$, and symmetrically). Hence,
under the standing amphichirality hypothesis,
\[
\mathrm{Sym}(K)=G
\iff m\in\mathrm{Sym}(K)\ \text{and}\ mr\in\mathrm{Sym}(K).
\]
Combining this with (a)--(b), and noting that the subgroups of $G$ containing at least one
of $m,mr$ are exactly $\{1,m\}$, $\{1,mr\}$, and $G$:
\begin{enumerate}[label=\textup{(\roman*)}]
\item If $c$ is a palindrome and has odd period, then $m,mr\in\mathrm{Sym}(K)$, so
$\mathrm{Sym}(K)=G$ and $K$ is fully amphichiral.
\item If $c$ has odd period but is not a palindrome, then $m\in\mathrm{Sym}(K)$ while
$mr\notin\mathrm{Sym}(K)$, so $\mathrm{Sym}(K)=\{1,m\}$ and $K$ is positive amphichiral.
\item If $c$ is a palindrome but does not have odd period, then $mr\in\mathrm{Sym}(K)$
while $m\notin\mathrm{Sym}(K)$, so $\mathrm{Sym}(K)=\{1,mr\}$ and $K$ is negative
amphichiral.
\end{enumerate}
These three cases are mutually exclusive and exhaust the amphichiral possibilities,
giving (i)--(iii).
\end{proof}

\noindent \Cref{thm:unique-amphi} settles every knot whose $3$-braid
representative has unique conjugacy class. By the Birman--Menasco classification
(\Cref{thm:BM}) the only remaining possibility, after the unknot and torus knots
excluded in \Cref{sec:partI}, is that $K$ is one of the special links carrying a
non-degenerate \emph{flype}, where two distinct conjugacy classes represent the same
knot. We treat this case now, using the flype apparatus (\Cref{def:flype},
\Cref{thm:BM,thm:invert}, and \Cref{lem:flype-twist}) introduced above.

The next lemma supplies the other half of \Cref{thm:B,thm:C}: it treats the amphichiral
knots arising from exception~\textup{(iii)} of \Cref{thm:BM}, and together with
\Cref{thm:flype-amphi} it yields case~\textup{(B)} of \Cref{thm:C} and the remaining half
of \Cref{thm:B}. The knot $6_3$, already pictured in \Cref{fig:braid63closed}, is the
smallest example, and further examples appear in the appendix.
\begin{lemma}[Two braid identities]\label{lem:flype-identities}
For every integer $k\ge2$,
\begin{align}
\sigma_2
\bigl(
\sigma_1\sigma_2^k\sigma_1^{-k}\sigma_2^{-1}
\bigr)
\sigma_2^{-1}
&=
\sigma_1^{k-1}\sigma_2^{-(k-1)}\sigma_1\sigma_2^{-1},
\label{eq:flype-identity-one}\\
\sigma_1
\bigl(
\sigma_1^k\sigma_2^2\sigma_1^{-k-1}\sigma_2^{-1}
\bigr)
\sigma_1^{-1}
&=
\sigma_1^k\sigma_2^{-1}\sigma_1\sigma_2^{-k}.
\label{eq:flype-identity-two}
\end{align}
In particular, the braid inside the parentheses on the left-hand side of
each identity is conjugate to the braid on the corresponding right-hand side.
\end{lemma}

\begin{proof}
Put $\Delta=\sigma_1\sigma_2\sigma_1=\sigma_2\sigma_1\sigma_2$. The braid relation gives $\Delta\sigma_1^j=\sigma_2^j\Delta$ and $\Delta\sigma_2^j=\sigma_1^j\Delta$ for every $j\in\mathbb Z$. We shall also use $\sigma_2\sigma_1=\Delta\sigma_2^{-1}$, $\sigma_1^{-2}\Delta=\sigma_2\sigma_1\sigma_2^{-1}$, and $\sigma_1^{-1}\Delta\sigma_1^{-1}=\sigma_2$, which follow immediately from the two expressions for $\Delta$.

For the first identity, we compute
\begin{align*}
&\sigma_2
 \bigl(\sigma_1\sigma_2^k\sigma_1^{-k}\sigma_2^{-1}\bigr)
 \sigma_2^{-1}
=\sigma_2\sigma_1\sigma_2^k\sigma_1^{-k}\sigma_2^{-2}
=\Delta\sigma_2^{k-1}\sigma_1^{-k}\sigma_2^{-2}\\
&\qquad
=\sigma_1^{k-1}\Delta\sigma_1^{-k}\sigma_2^{-2}
=\sigma_1^{k-1}\sigma_2^{-k}\Delta\sigma_2^{-2}
=\sigma_1^{k-1}\sigma_2^{-k}\sigma_1^{-2}\Delta\\
&\qquad
=\sigma_1^{k-1}\sigma_2^{-k}
 \bigl(\sigma_2\sigma_1\sigma_2^{-1}\bigr)
=\sigma_1^{k-1}\sigma_2^{-(k-1)}\sigma_1\sigma_2^{-1}.
\end{align*}

For the second identity, we similarly obtain
\begin{align*}
&\sigma_1
 \bigl(\sigma_1^k\sigma_2^2\sigma_1^{-k-1}\sigma_2^{-1}\bigr)
 \sigma_1^{-1}
=\sigma_1^{k+1}\sigma_2^2\sigma_1^{-k-1}\sigma_2^{-1}\sigma_1^{-1}
=\sigma_1^k\Delta\sigma_1^{-1}\sigma_2
 \sigma_1^{-k-1}\sigma_2^{-1}\sigma_1^{-1}\\
&\qquad
=\sigma_1^k\sigma_2^{-1}\Delta\sigma_2
 \sigma_1^{-k-1}\sigma_2^{-1}\sigma_1^{-1}
=\sigma_1^k\sigma_2^{-1}\sigma_1\Delta
 \sigma_1^{-k-1}\sigma_2^{-1}\sigma_1^{-1}
=\sigma_1^k\sigma_2^{-1}\sigma_1\sigma_2^{-k-1}
 \Delta\sigma_2^{-1}\sigma_1^{-1}\\
&\qquad
=\sigma_1^k\sigma_2^{-1}\sigma_1\sigma_2^{-k-1}
 \sigma_1^{-1}\Delta\sigma_1^{-1}
=\sigma_1^k\sigma_2^{-1}\sigma_1\sigma_2^{-k-1}\sigma_2
=\sigma_1^k\sigma_2^{-1}\sigma_1\sigma_2^{-k}.
\end{align*}
This proves both identities directly from the braid relation.
\end{proof}

\begin{lemma}[Ko--Lee flype criteria
{\cite{KoLee1999}; see also \cite[Lemma~1]{BM2008}}]
\label{lem:kolee-flype}
Let
\[
L=
\sigma_1^u\sigma_2^v\sigma_1^w\sigma_2^{\varepsilon},
\qquad \varepsilon\in\{\pm1\},
\]
be a flype-admissible $3$-braid.
\begin{enumerate}[label=\textup{(\alph*)}]
\item The displayed flype is non-degenerate if and only if $|v|\ge2$, the integers $u$, $v+\varepsilon$, and $w$ are pairwise distinct, and $u,w\notin\{0,\varepsilon,2\varepsilon\}$.
\item Suppose that $\varepsilon=-1$, so that $L$ admits a negative flype. Then the conjugacy class of $L$ contains a representative admitting a positive flype if and only if $u=1$, $w=1$, or $v=2$.
\end{enumerate}
\end{lemma}

\begin{lemma}[Amphichiral flype family]\label{lem:flype-family}
Suppose that $\beta_c$ is in alternating standard form and that
$K=\widehat{\beta_c}$ is a prime knot of braid index~$3$. Suppose further
that $K$ is amphichiral and that the conjugacy class of $\beta_c$ admits a
non-degenerate flype. Then, up to a rotation,
\[
c=(k,1,1,k)
\]
for some $k\ge2$. In particular, $c$ is a palindrome.
\end{lemma}

\begin{proof}
Because the conjugacy class of $\beta_c$ admits a non-degenerate flype,
\Cref{thm:invert} implies that $K$ is invertible. Since $K$ is also
amphichiral, it is in fact fully amphichiral: if either $m$ or $mr$ fixes
$K$, then multiplication by $r\in\mathrm{Sym}(K)$ shows that the other one
does as well. In particular, $mK=K$ as oriented knot types.

Mirroring a positive flype gives a negative flype and preserves
non-degeneracy. Therefore one of the two conjugacy classes of $3$-braids
representing $K$ contains a non-degenerate negative-flype representative
$L=\sigma_1^u\sigma_2^v\sigma_1^w\sigma_2^{-1}$. Its literal flype partner
and the conjugate flype-partner representative from \Cref{def:flype} are,
respectively,
\[
L^{\mathrm f}=\sigma_1^w\sigma_2^v\sigma_1^u\sigma_2^{-1},
\qquad
\widetilde L^{\mathrm f}=\sigma_1^u\sigma_2^{-1}\sigma_1^w\sigma_2^v,
\]
and these two braids are conjugate. The braids $L$ and $L^{\mathrm f}$
represent the two distinct conjugacy classes associated with the
non-degenerate flype.

The mirror $m(L)$ is a positive-flype representative of $mK=K$. Hence
$m(L)$ belongs to one of the two conjugacy classes represented by $L$ and
$L^{\mathrm f}$. Both classes contain a negative-flype representative,
namely $L$ or $L^{\mathrm f}$, respectively. Thus one of the two classes
contains representatives admitting both a negative and a positive flype.
Applying \Cref{lem:kolee-flype}(b) to the negative-flype representative in
that class gives $u=1$, $w=1$, or $v=2$.
If the criterion is applied to the class of $L^{\mathrm f}$, the roles of
$u$ and $w$ are interchanged, so the same three alternatives result.

Every $3$-braid representative of $K$ is minimal because $b(K)=3$.
Amphichirality and \eqref{eq:wmin-zero} therefore give
$e(L)=u+v+w-1=0$. We now consider the three alternatives above. Passing from
$L$ to $L^{\mathrm f}$ interchanges $u$ and $w$, so the cases $u=1$ and
$w=1$ are equivalent.

\smallskip
\noindent\emph{Case 1: $u=1$ or $w=1$.}
After interchanging $L$ and $L^{\mathrm f}$ if necessary, assume $u=1$.
Then $u+v+w-1=0$ gives $w=-v$. If $v=-k<0$, then
$k\ge2$ by \Cref{lem:kolee-flype}(a), and
\[
L=\sigma_1\sigma_2^{-k}\sigma_1^k\sigma_2^{-1}
=\beta_{(1,k,k,1)}.
\]
The word $(1,k,k,1)$ is a rotation of $(k,1,1,k)$.

If $v=k>0$, then $w=-k$. By \Cref{lem:kolee-flype}(a), the integers
$u$, $v-1$, and $w$ are pairwise distinct. The value $k=2$ would give
$u=v-1=1$, so $k\ge3$. By \eqref{eq:flype-identity-one},
\[
L=\sigma_1\sigma_2^k\sigma_1^{-k}\sigma_2^{-1}
\sim\sigma_1^{k-1}\sigma_2^{-(k-1)}\sigma_1\sigma_2^{-1}
=\beta_{(k-1,k-1,1,1)}.
\]
Here $k-1\ge2$, and $(k-1,k-1,1,1)$ is a rotation of
$(k-1,1,1,k-1)$.

\smallskip
\noindent\emph{Case 2: $v=2$.}
The equality $u+v+w-1=0$ gives $u+w=-1$. For a non-degenerate negative
flype, \Cref{lem:kolee-flype}(a) gives $u,w\notin\{0,-1,-2\}$ and
$u,w\ne v-1=1$.
It follows that exactly one of $u,w$ is a positive integer at least~$2$.
After interchanging $L$ and $L^{\mathrm f}$ if necessary, write $u=k\ge2$
and $w=-k-1$. By \eqref{eq:flype-identity-two},
\[
L=\sigma_1^k\sigma_2^2\sigma_1^{-k-1}\sigma_2^{-1}
\sim\sigma_1^k\sigma_2^{-1}\sigma_1\sigma_2^{-k}
=\beta_{(k,1,1,k)}.
\]

Thus at least one of the two conjugacy classes representing $K$ contains
an alternating standard form whose word $d$ is a rotation of
$(k,1,1,k)$ for some $k\ge2$. By \Cref{lem:dihedral}(a), the standard
form with word $\rho(d)$ represents $mK=K$. Since $k\ge2$, the word
$\rho(d)$ is not an even rotation of $d$; hence
\Cref{thm:conjugacy_3braids} shows that $\beta_{\rho(d)}$ lies in the
other conjugacy class. Consequently, the standard words occurring in
both conjugacy classes are rotations of $(k,1,1,k)$.

Finally, $\beta_c$ belongs to one of these two classes. Applying
\Cref{thm:conjugacy_3braids} within that class shows that $c$ is an even
rotation of either $d$ or $\rho(d)$. Therefore $c$ is, without any parity
restriction, a rotation of $(k,1,1,k)$. This word is palindromic up to
rotation, completing the proof.
\end{proof}

\begin{theorem}[Classification: non-degenerate flype]
\label{thm:flype-amphi}
Retain hypothesis~\textup{(a)} of \Cref{thm:unique-amphi}, but replace
hypothesis~\textup{(b)} by the assumption that the conjugacy class of
$\beta_c$ admits a non-degenerate flype. Then $K$ is amphichiral if and
only if $c$ is a rotation of $(k,1,1,k)$ with $k\ge2$. In that case $K$
is fully amphichiral.
\end{theorem}

\begin{proof}
If $K$ is amphichiral, \Cref{lem:flype-family} gives that $c$ is a
rotation of $(k,1,1,k)$ with $k\ge2$.

Conversely, suppose that $c$ is a rotation of $(k,1,1,k)$. Since $c$ is
a palindrome in the sense of \eqref{eq:palindrome}, there is an integer
$s\ge0$ such that
\[
\tau\rho^{2s}(c)=c.
\]
Because $\varphi(\tau\rho^{2s})=mr$, \Cref{lem:dihedral} gives
$K=mr(K)$, so $K$ is amphichiral. Finally, a non-degenerate flype makes
$K$ invertible by \Cref{thm:invert}. Thus $r,mr\in\mathrm{Sym}(K)$, and
hence
\[
m=r\cdot mr\in\mathrm{Sym}(K).
\]
Therefore $\mathrm{Sym}(K)=G$, and $K$ is fully amphichiral.
\end{proof}

\noindent We now prove \Cref{thm:B,thm:C} together; both rest on the same
structural dichotomy.

\begin{proof}[Proof of \Cref{thm:B,thm:C}]
Since $b(K)=3$, the knot $K$ is neither the unknot nor a closed $2$-braid;
in particular, it is not a $(2,k)$-torus knot. Therefore the first two
exceptions in the Birman--Menasco classification (\Cref{thm:BM}) do not
occur. Exactly one of the following alternatives holds:
\begin{enumerate}[label=\textup{(\roman*)}]
\item $K$ is represented by a unique conjugacy class of $3$-braids; or
\item $K$ is represented by two conjugacy classes related by a
non-degenerate flype.
\end{enumerate}
In the first alternative, \Cref{thm:unique-amphi} gives the criterion and the three symmetry types in \Cref{thm:C}(A). In the second alternative, \Cref{thm:flype-amphi} gives \Cref{thm:C}(B). For the necessity direction of \Cref{thm:B} in the second alternative, \Cref{thm:flype-amphi} shows that if $K$ is amphichiral, then $c$ is a rotation of $(k,1,1,k)$ for some $k\ge2$. Since every such rotation is a palindrome in the sense of \eqref{eq:palindrome}, the required condition follows. Conversely, the sufficiency direction of \Cref{thm:B} follows directly from \Cref{lem:dihedral}, without any assumption on the number of conjugacy classes. If $c$ is a palindrome, then $\tau\rho^{2s}(c)=c$ for some integer $s\ge0$. Since $\varphi(\tau\rho^{2s})=mr$, \Cref{lem:dihedral} gives $K=mr(K)$. If instead $c$ has odd period $p$, then $\rho^p(c)=c$, and since $p$ is odd, $\varphi(\rho^p)=m^p=m$. Thus \Cref{lem:dihedral} gives $K=m(K)$. Hence in either case $K$ is amphichiral, establishing the converse implication. The two alternatives are mutually exclusive and exhaustive, which completes both proofs.
\end{proof}

\begin{remark}\label{rmk:case-B-and-data}
Case~B is the unique situation in which a palindromic $c$ of even
period yields full amphichirality rather than negative amphichirality
alone; this is why the negative-only case in
\Cref{thm:C}\textup{(A)(iii)} is confined to the unique-class setting. By \Cref{thm:C}, a knot is positive but not fully amphichiral exactly
when $c$ has odd period without being a palindrome; among all prime amphichiral knots
of braid index~$3$ with crossing number at most~$14$, only $14a_{18676}$
satisfies this condition, making it a rare case within this range
\textup{(}\Cref{tab:bi3-cn14}\textup{)}. The classification of
\Cref{thm:B,thm:C} agrees with the tabulated data of
\Cref{app:bi3-cn12,app:bi3-cn14} for every prime amphichiral $3$-braid knot with
crossing number up to~$14$.
\end{remark}

\newpage

\bibliographystyle{alpha}
\bibliography{mybib}

@article {BM2008,
    AUTHOR = {Birman, Joan S. and Menasco, William W.},
     TITLE = {A note on closed 3-braids},
   JOURNAL = {Commun. Contemp. Math.},
  FJOURNAL = {Communications in Contemporary Mathematics},
    VOLUME = {10},
      YEAR = {2008},
     PAGES = {1033--1047},
      ISSN = {0219-1997,1793-6683},
   MRCLASS = {57M25 (57M27)},
  MRNUMBER = {2468377},
MRREVIEWER = {Martin\ Scharlemann},
       DOI = {10.1142/S0219199708003150},
       URL = {https://doi.org/10.1142/S0219199708003150},
}

@book {KasselTuraev,
    AUTHOR = {Kassel, Christian and Turaev, Vladimir},
     TITLE = {Braid groups},
    SERIES = {Graduate Texts in Mathematics},
    VOLUME = {247},
      NOTE = {With the graphical assistance of Olivier Dodane},
 PUBLISHER = {Springer, New York},
      YEAR = {2008},
     PAGES = {xii+340},
      ISBN = {978-0-387-33841-5},
   MRCLASS = {20F36 (20C08 20F10 20F60 20M05 55R80 57M25 57M50)},
  MRNUMBER = {2435235},
MRREVIEWER = {Stephen\ P.\ Humphries},
       DOI = {10.1007/978-0-387-68548-9},
       URL = {https://doi.org/10.1007/978-0-387-68548-9},
}

@misc{knotinfo,
      title={KnotInfo: Table of Knot Invariants}, 
      author={J.~C. Cha and C.~Livingston},
      year={2004},
      note={\url{https://knotinfo.org/search-general.php}, 
accessed June 2026}
      }

@article {BM1993,
    AUTHOR = {Birman, Joan S. and Menasco, William W.},
     TITLE = {Studying links via closed braids. {VI}. {A} nonfiniteness
              theorem},
   JOURNAL = {Pacific J. Math.},
  FJOURNAL = {Pacific Journal of Mathematics},
    VOLUME = {156},
      YEAR = {1992},
    NUMBER = {2},
     PAGES = {265--285},
      ISSN = {0030-8730,1945-5844},
   MRCLASS = {57M25 (20F36)},
  MRNUMBER = {1186805},
MRREVIEWER = {Hugh\ Reynolds\ Morton},
       URL = {http://projecteuclid.org/euclid.pjm/1102634977},
}

@article {DasbachLowrance2011,
    AUTHOR = {Dasbach, Oliver T. and Lowrance, Adam M.},
     TITLE = {Turaev genus, knot signature, and the knot homology
              concordance invariants},
   JOURNAL = {Proc. Amer. Math. Soc.},
  FJOURNAL = {Proceedings of the American Mathematical Society},
    VOLUME = {139},
      YEAR = {2011},
    NUMBER = {7},
     PAGES = {2631--2645},
      ISSN = {0002-9939,1088-6826},
   MRCLASS = {57M25 (57M27)},
  MRNUMBER = {2784832},
MRREVIEWER = {Dale\ P. O. Rolfsen},
       DOI = {10.1090/S0002-9939-2010-10698-6},
       URL = {https://doi.org/10.1090/S0002-9939-2010-10698-6},
}

@article {DP2013,
    AUTHOR = {Dynnikov, I. A. and Prasolov, M. V.},
     TITLE = {Bypasses for rectangular diagrams. {A} proof of the {J}ones
              conjecture and related questions},
   JOURNAL = {Trans. Moscow Math. Soc.},
  FJOURNAL = {Transactions of the Moscow Mathematical Society},
      YEAR = {2013},
     PAGES = {97--144},
      ISSN = {0077-1554,1547-738X},
   MRCLASS = {57M25},
  MRNUMBER = {3235791},
       DOI = {10.1090/s0077-1554-2014-00210-7},
       URL = {https://doi.org/10.1090/s0077-1554-2014-00210-7},
}

@article {KoLee1999,
    AUTHOR = {Ko, Ki Hyoung and Lee, Sang Jin},
     TITLE = {Flypes of closed {$3$}-braids in the standard contact space},
   JOURNAL = {J. Korean Math. Soc.},
  FJOURNAL = {Journal of the Korean Mathematical Society},
    VOLUME = {36},
      YEAR = {1999},
    NUMBER = {1},
     PAGES = {51--71},
      ISSN = {0304-9914,2234-3008},
   MRCLASS = {57M25},
  MRNUMBER = {1669133},
MRREVIEWER = {Rostislav\ Matveyev},
}

@article {StoimenowMFW,
    AUTHOR = {Stoimenow, A.},
     TITLE = {On the crossing number of positive knots and braids and braid
              index criteria of {J}ones and {M}orton-{W}illiams-{F}ranks},
   JOURNAL = {Trans. Amer. Math. Soc.},
  FJOURNAL = {Transactions of the American Mathematical Society},
    VOLUME = {354},
      YEAR = {2002},
    NUMBER = {10},
     PAGES = {3927--3954},
      ISSN = {0002-9947,1088-6850},
   MRCLASS = {57M25 (20F36)},
  MRNUMBER = {1926860},
MRREVIEWER = {Stephen\ P.\ Humphries},
       DOI = {10.1090/S0002-9947-02-03022-2},
       URL = {https://doi.org/10.1090/S0002-9947-02-03022-2},
}

\appendix
\crefalias{section}{appendix}

\begin{landscape}
\section[Prime 3-braid-index knot table (crossing number <= 12)]%
        {Prime $3$-braid-index knot table \cite{knotinfo} (crossing number $\le 12$)}
\label{app:bi3-cn12}
\renewcommand{\arraystretch}{1.25}
\begin{longtable}{c c c l c}
\caption{Prime amphichiral knots of braid index $3$ with crossing number at most $12$.}
\label{tab:bi3-cn12}\\
\toprule
\textbf{Knot} & \textbf{Word $c$} & \textbf{Braid word} &
\textbf{Classification} & \textbf{Symmetry type}\\
\hline
\endfirsthead
\multicolumn{5}{l}{\small\upshape \Cref{tab:bi3-cn12} continued from previous page}\\
\hline
\textbf{Knot} & \textbf{Word $c$} & \textbf{Braid word} &
\textbf{Classification} & \textbf{Symmetry type}\\
\hline
\endhead
\hline
\endfoot
\texttt{4a\_1} & $(1,1,1,1)$ & \texttt{[1,-2,1,-2]} & palindrome \& $p=1$ & fully\\
\texttt{6a\_3} & $(2,1,1,2)$ & \texttt{[1,1,-2,1,-2,-2]} & palindrome (flype) & fully\\
\texttt{8a\_9} & $(3,1,1,3)$ & \texttt{[1,1,1,-2,1,-2,-2,-2]} & palindrome (flype) & fully\\
\texttt{8a\_17} & $(2,1,1,1,1,2)$ & \texttt{[1,1,-2,1,-2,1,-2,-2]} & palindrome & negative\\
\texttt{8a\_18} & $(1,1,1,1,1,1,1,1)$ & \texttt{[1,-2,1,-2,1,-2,1,-2]} & palindrome \& $p=1$ & fully\\
\texttt{10a\_17} & $(4,1,1,4)$ & \texttt{[1,1,1,1,-2,1,-2,-2,-2,-2]} & palindrome (flype) & fully\\
\texttt{10a\_79} & $(3,2,2,3)$ & \texttt{[1,1,1,-2,-2,1,1,-2,-2,-2]} & palindrome & negative\\
\texttt{10a\_99} & $(2,1,2,2,1,2)$ & \texttt{[1,1,-2,1,1,-2,-2,1,-2,-2]} & palindrome \& $p=3$ & fully\\
\texttt{10a\_109} & $(1,2,2,2,2,1)$ & \texttt{[1,-2,-2,1,1,-2,-2,1,1,-2]} & palindrome & negative\\
\texttt{10a\_118} & $(1,1,2,1,1,2,1,1)$ & \texttt{[1,-2,1,1,-2,1,-2,-2,1,-2]} & palindrome & negative\\
\texttt{10a\_123} & $(1,1,1,1,1,1,1,1,1,1)$ & \texttt{[1,-2,1,-2,1,-2,1,-2,1,-2]} & palindrome \& $p=1$ & fully\\
\texttt{12a\_819} & $(1,3,1,1,1,1,3,1)$ & \texttt{[1,-2,-2,-2,1,-2,1,-2,1,1,1,-2]} & palindrome & negative\\
\texttt{12a\_1209} & $(3,1,1,1,1,1,1,3)$ & \texttt{[1,1,1,-2,1,-2,1,-2,1,-2,-2,-2]} & palindrome & negative\\
\texttt{12a\_1211} & $(2,1,1,1,1,1,1,1,1,2)$ & \texttt{[1,1,-2,1,-2,1,-2,1,-2,1,-2,-2]} & palindrome & negative\\
\texttt{12a\_1218} & $(1,1,4,4,1,1)$ & \texttt{[1,-2,1,1,1,1,-2,-2,-2,-2,1,-2]} & palindrome & negative\\
\texttt{12a\_1225} & $(1,1,2,2,2,2,1,1)$ & \texttt{[1,-2,1,1,-2,-2,1,1,-2,-2,1,-2]} & palindrome & negative\\
\texttt{12a\_1229} & $(2,1,1,2,2,1,1,2)$ & \texttt{[1,1,-2,1,-2,-2,1,1,-2,1,-2,-2]} & palindrome & negative\\
\texttt{12a\_1249} & $(1,2,1,1,1,1,1,1,2,1)$ & \texttt{[1,-2,-2,1,-2,1,-2,1,-2,1,1,-2]} & palindrome & negative\\
\texttt{12a\_1254} & $(2,1,3,3,1,2)$ & \texttt{[1,1,-2,1,1,1,-2,-2,-2,1,-2,-2]} & palindrome & negative\\
\texttt{12a\_1260} & $(1,3,2,2,3,1)$ & \texttt{[1,-2,-2,-2,1,1,-2,-2,1,1,1,-2]} & palindrome & negative\\
\texttt{12a\_1273} & $(1,5,5,1)$ & \texttt{[1,-2,-2,-2,-2,-2,1,1,1,1,1,-2]} & palindrome (flype) & fully\\
\texttt{12a\_1288} & $(3,3,3,3)$ & \texttt{[1,1,1,-2,-2,-2,1,1,1,-2,-2,-2]} & palindrome \& $p=1$ & fully\\
\bottomrule
\end{longtable}

\smallskip
\noindent\footnotesize
Each braid word is the standard form
$\beta_c=\sigma_{1}^{a_1}\sigma_{2}^{-b_1}\cdots\sigma_{1}^{a_n}\sigma_{2}^{-b_n}$
read directly from the word $c=(a_1,b_1,\dots,a_n,b_n)$, so every
$\sigma_{1}$-syllable carries a positive exponent and every $\sigma_{2}$-syllable a
negative one (\Cref{def:standard_form}); it represents $K$ as a closed
braid. ``(flype)'' marks the $(k,1,1,k)$ family, whose full amphichirality is
supplied by \Cref{thm:flype-amphi}; all other entries are determined by
\Cref{thm:unique-amphi}, and the combined classification is given by \Cref{thm:B,thm:C}.
Every entry is alternating, in accordance with
\Cref{thm:amphi-implies-alt}: no non-alternating amphichiral knot of braid
index~$3$ exists, in this or any crossing-number range. The table therefore lists
\emph{all prime} amphichiral knots of braid index~$3$ with crossing number at most~$12$ and
\Cref{thm:B,thm:C} apply to each. The entry \texttt{6a\_3}~($=6_3$), the smallest example in the non-degenerate flype family $(k,1,1,k)$, is depicted as a closed braid in \Cref{fig:braid63closed}.
\end{landscape}

\begin{landscape}
\section[Prime 3-braid-index knot table (crossing number 14)]%
        {Prime $3$-braid-index knot table \cite{knotinfo} (crossing number $14$)}
\label{app:bi3-cn14}
\renewcommand{\arraystretch}{1.25}
\footnotesize
\begin{longtable}{c c c l c}
\caption{Prime amphichiral knots of braid index $3$ with crossing number $14$.}
\label{tab:bi3-cn14}\\
\toprule
\textbf{Knot} & \textbf{Word $c$} & \textbf{Braid word} &
\textbf{Classification} & \textbf{Symmetry type}\\
\hline
\endfirsthead
\multicolumn{5}{l}{\small\upshape \Cref{tab:bi3-cn14} continued from previous page}\\
\hline
\textbf{Knot} & \textbf{Word $c$} & \textbf{Braid word} &
\textbf{Classification} & \textbf{Symmetry type}\\
\hline
\endhead
\hline
\endfoot
\texttt{14a\_13528} & $(1,1,4,1,1,4,1,1)$ & \texttt{[1,-2,1,1,1,1,-2,1,-2,-2,-2,-2,1,-2]} & palindrome & negative\\
\texttt{14a\_18218} & $(5,2,2,5)$ & \texttt{[1,1,1,1,1,-2,-2,1,1,-2,-2,-2,-2,-2]} & palindrome & negative\\
\texttt{14a\_18301} & $(3,2,2,2,2,3)$ & \texttt{[1,1,1,-2,-2,1,1,-2,-2,1,1,-2,-2,-2]} & palindrome & negative\\
\texttt{14a\_18306} & $(1,1,1,3,1,1,3,1,1,1)$ & \texttt{[1,-2,1,-2,-2,-2,1,-2,1,1,1,-2,1,-2]} & palindrome & negative\\
\texttt{14a\_18362} & $(2,3,2,2,3,2)$ & \texttt{[1,1,-2,-2,-2,1,1,-2,-2,1,1,1,-2,-2]} & palindrome \& $p=3$ & fully\\
\texttt{14a\_18408} & $(3,4,4,3)$ & \texttt{[1,1,1,-2,-2,-2,-2,1,1,1,1,-2,-2,-2]} & palindrome & negative\\
\texttt{14a\_18599} & $(1,3,2,1,1,2,3,1)$ & \texttt{[1,-2,-2,-2,1,1,-2,1,-2,-2,1,1,1,-2]} & palindrome & negative\\
\texttt{14a\_18636} & $(1,2,4,4,2,1)$ & \texttt{[1,-2,-2,1,1,1,1,-2,-2,-2,-2,1,1,-2]} & palindrome & negative\\
\texttt{14a\_18676} & $(1,2,4,1,2,4)$ & \texttt{[1,-2,-2,1,1,1,1,-2,1,1,-2,-2,-2,-2]} & $p=3$ & positive\\
\texttt{14a\_18735} & $(1,2,2,2,2,2,2,1)$ & \texttt{[1,-2,-2,1,1,-2,-2,1,1,-2,-2,1,1,-2]} & palindrome & negative\\
\texttt{14a\_18774} & $(2,2,1,2,2,1,2,2)$ & \texttt{[1,1,-2,-2,1,-2,-2,1,1,-2,1,1,-2,-2]} & palindrome & negative\\
\texttt{14a\_18986} & $(2,1,4,4,1,2)$ & \texttt{[1,1,-2,1,1,1,1,-2,-2,-2,-2,1,-2,-2]} & palindrome & negative\\
\texttt{14a\_19162} & $(1,4,2,2,4,1)$ & \texttt{[1,-2,-2,-2,-2,1,1,-2,-2,1,1,1,1,-2]} & palindrome & negative\\
\texttt{14a\_19298} & $(1,6,6,1)$ & \texttt{[1,-2,-2,-2,-2,-2,-2,1,1,1,1,1,1,-2]} & palindrome (flype) & fully\\
\texttt{14a\_19385} & $(2,1,1,2,1,1,2,1,1,2)$ & \texttt{[1,1,-2,1,-2,-2,1,-2,1,1,-2,1,-2,-2]} & palindrome & negative\\
\texttt{14a\_19450} & $(3,1,2,1,1,2,1,3)$ & \texttt{[1,1,1,-2,1,1,-2,1,-2,-2,1,-2,-2,-2]} & palindrome & negative\\
\texttt{14a\_19470} & $(1,1,1,1,1,1,1,1,1,1,1,1,1,1)$ & \texttt{[1,-2,1,-2,1,-2,1,-2,1,-2,1,-2,1,-2]} & palindrome \& $p=1$ & fully\\
\texttt{14a\_19472} & $(1,1,3,1,1,1,1,3,1,1)$ & \texttt{[1,-2,1,1,1,-2,1,-2,1,-2,-2,-2,1,-2]} & palindrome \& $p=5$ & fully\\
\texttt{14a\_19473} & $(1,1,1,1,1,2,2,1,1,1,1,1)$ & \texttt{[1,-2,1,-2,1,-2,-2,1,1,-2,1,-2,1,-2]} & palindrome & negative\\
\texttt{14a\_19476} & $(3,1,1,1,1,1,1,1,1,3)$ & \texttt{[1,1,1,-2,1,-2,1,-2,1,-2,1,-2,-2,-2]} & palindrome & negative\\
\texttt{14a\_19487} & $(1,1,1,2,1,1,1,1,2,1,1,1)$ & \texttt{[1,-2,1,-2,-2,1,-2,1,-2,1,1,-2,1,-2]} & palindrome & negative\\
\texttt{14a\_19498} & $(1,1,2,3,3,2,1,1)$ & \texttt{[1,-2,1,1,-2,-2,-2,1,1,1,-2,-2,1,-2]} & palindrome & negative\\
\texttt{14a\_19517} & $(1,2,1,2,1,1,2,1,2,1)$ & \texttt{[1,-2,-2,1,-2,-2,1,-2,1,1,-2,1,1,-2]} & palindrome \& $p=5$ & fully\\
\texttt{14a\_19526} & $(1,1,2,1,2,2,1,2,1,1)$ & \texttt{[1,-2,1,1,-2,1,1,-2,-2,1,-2,-2,1,-2]} & palindrome & negative\\
\bottomrule
\end{longtable}
\normalsize

\smallskip
\noindent\footnotesize
As in \Cref{tab:bi3-cn12}, each braid word is the standard form $\beta_c$
read directly from $c$ (\Cref{def:standard_form}). Among all prime amphichiral knots
of braid index~$3$ with crossing number up to~$14$, the entry \texttt{14a\_18676} is the
only positive amphichiral knot; it realizes the ``odd period, not a palindrome'' case of
\Cref{thm:C}, and is depicted as a closed braid in \Cref{fig:braid14a18676closed}.
Again every prime amphichiral knot of braid index~$3$ with crossing number~$14$
is alternating (\Cref{thm:amphi-implies-alt}), so the prime-knot list is complete.
\end{landscape}

\begin{landscape}
\section[Prime 5-braid-index knot table (crossing number <= 10)]%
        {Prime $5$-braid-index knot table \cite{knotinfo} (crossing number $\le 10$)}
\label{app:bi5-cn10}
\renewcommand{\arraystretch}{1.25}
\begin{longtable}{c c c}
\caption{Prime amphichiral knots of braid index $5$ with crossing number at most $10$.}
\label{tab:bi5-cn10}\\
\toprule
\textbf{Knot} & \textbf{Braid word} & \textbf{Symmetry type}\\
\hline
\endfirsthead
\multicolumn{3}{l}{\small\upshape \Cref{tab:bi5-cn10} continued from previous page}\\
\hline
\textbf{Knot} & \textbf{Braid word} & \textbf{Symmetry type}\\
\hline
\endhead
\hline
\endfoot
\texttt{8a\_3}   & \texttt{[1,1,2,-1,-3,2,-3,-4,3,-4]} & fully\\
\texttt{8a\_12}  & \texttt{[1,-2,1,3,-2,-4,3,-4]} & fully\\
\texttt{10a\_33} & \texttt{[1,1,2,-1,2,-3,2,-3,-3,-4,3,-4]} & fully\\
\texttt{10a\_37} & \texttt{[1,1,1,2,-1,-3,2,-3,-4,3,-4,-4]} & fully\\
\texttt{10a\_43} & \texttt{[1,1,-2,1,3,-2,-4,3,-4,-4]} & fully\\
\texttt{10a\_45} & \texttt{[1,-2,1,-2,3,-2,3,-4,3,-4]} & fully\\
\texttt{10a\_81} & \texttt{[1,1,-2,1,3,2,2,-4,-3,-3,-3,-4]} & negative\\
\texttt{10a\_88} & \texttt{[1,-2,1,3,-2,3,-2,-4,3,-4]} & negative\\
\texttt{10a\_115}& \texttt{[1,-2,1,3,2,2,-4,-3,2,-3,-3,-4]} & negative\\
\bottomrule
\end{longtable}
\noindent{All prime amphichiral knots with crossing number at most $10$ have braid
index $3$ or $5$.}
\end{landscape}

\end{document}